\documentclass[11pt]{amsart}
\usepackage{amssymb}
\newtheorem{thm}{Theorem}[section]

\newtheorem{lem}[thm]{Lemma}
\newtheorem{cor}[thm]{Corollary}
\newtheorem{remk}[thm]{Remark}

\newtheorem{defn}[thm]{Definition}

\newenvironment{tightcenter}{%
  \setlength\topsep{0pt}
  \setlength\parskip{0pt}
  \begin{center}
}{%
  \end{center}
}

\newcommand{\al}{\alpha}
\newcommand{\be}{\beta}
\newcommand{\ga}{\gamma}

\newcommand{\rig}{\rightarrow}
\newcommand{\mrig}{\mathrel{-\!\!\!\!\!\rightarrow}}
\newcommand{\Rig}{\Rightarrow}

\newcommand{\bcdw}{\mathbin{\boldsymbol\cdot}}

\newcommand{\seteq}{\mathrel{\mbox{\,\textup{:}\!}=\nolinebreak }\,}

\newcommand{\sbA}{{\boldsymbol{A}}}
\newcommand{\sbB}{{\boldsymbol{B}}}
\newcommand{\sbC}{{\boldsymbol{C}}}
\newcommand{\sbD}{{\boldsymbol{D}}}
\newcommand{\sbE}{{\boldsymbol{E}}}

\newcommand{\sbG}{{\boldsymbol{G}}}

\newcommand{\sbS}{{\boldsymbol{S}}}

\newcommand{\sbX}{{\boldsymbol{X}}}
\newcommand{\sbY}{{\boldsymbol{Y}}}

\newcommand{\ov}{\overline}

\newcommand{\leibniz}{\boldsymbol{\varOmega}}

\newcommand{\Alg}[1]{{\boldsymbol{#1}}}
\newcommand{\Logic}[1]{{\mathbf{#1}}}

\newcommand{\Class}[1]{{\mathsf{#1}}}

\newcommand{\bdot}{\mathbin{\boldsymbol{\cdot}}}

\DeclareMathOperator{\Vop}{\mathbb{V}}
\DeclareMathOperator{\Qop}{\mathbb{Q}}

\begin{document}
\title[Varieties of De Morgan Monoids]{Varieties of De Morgan Monoids:\\Minimality and Irreducible Algebras}
\author{T.\ Moraschini}
\address{Institute of Computer Science, Academy of Sciences of the Czech Republic, Pod Vod\'{a}renskou v\v{e}\v{z}\'{i} 2, 182 07 Prague 8, Czech Republic.}
\email{moraschini@cs.cas.cz}
\author{J.G.\ Raftery}
\address{Department of Mathematics and Applied Mathematics,
 University of Pretoria,
 Private Bag X20, Hatfield,
 Pretoria 0028, South Africa}
\email{{james.raftery@up.ac.za}}
\author{J.J.\ Wannenburg}
\address{Department of Mathematics and Applied Mathematics,
 University of Pretoria,
 Private Bag X20, Hatfield,
 Pretoria 0028, and DST-NRF Centre of Excellence in Mathematical and Statistical Sciences (CoE-MaSS), South Africa}
\email{{jamie.wannenburg@up.ac.za}}
\vspace{1mm}
\vspace{1mm}
\thanks{This work received funding from the European Union's Horizon 2020 research and innovation programme under the Marie Sklodowska-Curie grant
agreement No~689176 (project ``Syntax Meets Semantics: Methods, Interactions, and Connections in Substructural logics").
The first author acknowledges project CZ.02.2.69/0.0/0.0/17\_050/0008361, OPVVV M\v{S}MT, MSCA-IF Lidsk\'{e} zdroje v teoretick\'{e} informatice,
and project GJ15-07724Y of the Czech Science Foundation.  The second author was supported in part by the National Research Foundation of South Africa (UID 85407).  The third author was supported by the DST-NRF Centre of Excellence in Mathematical and Statistical Sciences (CoE-MaSS), South Africa.  Opinions expressed and conclusions arrived at are those of the authors and are not necessarily to be attributed to the CoE-MaSS}

\begin{abstract}
It is proved that every finitely subdirectly irreducible De Morgan monoid $\sbA$ (with neutral element $e$)
is either (i)~a Sugihara chain in which $e$ covers $\neg e$ or (ii)~the union of an interval subalgebra $[\neg a,a]$
and two chains of idempotents, $(\neg a]$ and $[a)$, where $a=(\neg e)^2$.  In the latter case, the variety
generated by $[\neg a,a]$ has no nontrivial idempotent member, and $\sbA/[\neg a)$ is a Sugihara chain in
which $\neg e=e$.  It is also proved that there are just four minimal varieties of De Morgan monoids.
These findings are then used to simplify the proof of a description (due to K.~\'{S}wirydowicz) of the lower
part of the subvariety lattice of relevant algebras.  The results throw light on the models and the axiomatic
extensions of fundamental relevance logics.
\end{abstract}

\maketitle

\makeatletter
\renewcommand{\labelenumi}{\text{(\theenumi)}}
\renewcommand{\theenumi}{\roman{enumi}}
\renewcommand{\theenumii}{\roman{enumii}}
\renewcommand{\labelenumii}{\text{(\theenumii)}}
\renewcommand{\p@enumii}{\theenumi(\theenumii)}
\makeatother

{\allowdisplaybreaks


\section{Introduction}\label{introduction}

De Morgan monoids are commutative monoids with a residuated distributive lattice order and a compatible antitone involution
$\neg$,
where
$a\leqslant a^2$
for all elements $a$.  They form a variety, $\mathsf{DMM}$.

The explicit study of residuated lattices
goes back to Ward and Dilworth \cite{WD39} and has
older
antecedents
(see the citations in \cite{BR97,GJKO07,Gri85}).
Much of the interest in De Morgan monoids
stems, however, from
their connection with relevance logic, discovered by Dunn \cite{Dun66} and
recounted briefly below in Section~\ref{relevance logic}
(where further references are supplied).
A key fact, for our purposes, is that
the axiomatic extensions of Anderson and Belnap's logic $\mathbf{R}^\mathbf{t}$
and the varieties of De Morgan monoids form anti-isomorphic lattices, and the
latter are susceptible to the methods of
universal algebra.

Slaney \cite{Sla85,Sla89} showed that the free $0$--generated De Morgan monoid is finite, and that there are only seven
non-isomorphic subdirectly irreducible $0$--generated De Morgan monoids.  No similarly comprehensive
classification is available in
the $1$--generated case, however, where the algebras may already be infinite.
In 1996, Urquhart \cite[p.\,263]{Urq96} observed that
``[t]he algebraic theory of relevant logics is relatively unexplored, particularly by comparison with the field
of algebraic modal logic.''
Acquiescing in a paper of 2001, Dunn and Restall \cite[Sec.~3.5]{DR01} wrote:
``Not as much is known about the algebraic properties of De Morgan monoids as one would like.''
These remarks pre-date many recent papers on residuated lattices---see the bibliography of \cite{GJKO07}, for instance.
But the latter have concentrated mainly on varieties incomparable with $\mathsf{DMM}$ (e.g., Heyting and MV-algebras),
larger than $\mathsf{DMM}$ (e.g.,
full Lambek algebras) or smaller (e.g., Sugihara monoids).

A De Morgan monoid $\sbA$, with neutral element $e$, is said to be \emph{idempotent\/} or \emph{anti-idempotent\/}
if it satisfies $x^2=x$ or $x\leqslant (\neg e)^2$, respectively.  The idempotent De Morgan monoids
are the aforementioned \emph{Sugihara monoids}, and their structure is very well understood.
Anti-idempotence is equivalent to the demand that no nontrivial idempotent algebra belongs to the variety
generated by $\sbA$ (Corollary~\ref{sug cor}), hence the terminology.

It is well known that a De Morgan monoid is finitely subdirectly irreducible iff the element $e$
is join-prime.  The first main result of this paper shows
that any such De Morgan monoid $\sbA$ is either (i)~a totally ordered Sugihara monoid in which $e$ covers $\neg e$ or
(ii)~the union of an interval subalgebra $[\neg a,a]$ and
two chains of idempotent elements, $(\neg a]$ and $[a)$, where $a=
(\neg e)^2$.
In the latter case, the anti-idempotent subalgebra is the $e$--class of a congruence $\theta$ such that $\sbA/\theta$
is a totally ordered Sugihara monoid in which $\neg e=e$, and all other $\theta$--classes are singletons.
(See Theorem~\ref{lollipop 2} and Remark~\ref{interpretation}.)

Subalgebra structure aside, another measure of the
complexity of a De Morgan monoid $\sbA$ is the height, within the subvariety lattice of $\mathsf{DMM}$, of the
variety
generated by $\sbA$.
Accordingly, the present paper initiates an analysis of the lattice of varieties of De Morgan monoids.
We prove that such a variety consists of Sugihara monoids iff it omits a certain pair
of four-element algebras (Theorem~\ref{omit c4 d4}).  This implies that
$\mathsf{DMM}$ has just four minimal subvarieties, all of which are
finitely generated (Theorem~\ref{atoms}).  The covers of these atoms are investigated in a sequel paper \cite{MRW2}.

For philosophical reasons, relevance logic also emphasizes a system called $\mathbf{R}$, which
lacks the so-called Ackermann truth constant $\mathbf{t}$ (corresponding to the
neutral element of a De Morgan monoid).  The logic $\mathbf{R}$ is algebraized by the variety $\mathsf{RA}$ of
\emph{relevant algebras}.
\'{S}wirydowicz \cite{Swi95} has described the bottom of the subvariety lattice of $\mathsf{RA}$.
We simplify the proof of his result (see Theorem~\ref{thm:minimalvarietiesRA}), using our
analysis (from Section~\ref{structure section}) of the
subvarieties of $\mathsf{DMM}$.

These findings have implications for the extension lattices of both $\mathbf{R}$ and $\mathbf{R}^\mathbf{t}$.
For instance, \'{S}wirydowicz's theorem has been applied recently to show that no consistent axiomatic
extension of $\mathbf{R}$ is structurally complete, except for classical propositional logic \cite{RS}.  The
situation for $\mathbf{R}^\mathbf{t}$ is very different and is the subject of ongoing investigation by the
present authors.  (Regarding fragments of $\mathbf{R}^\mathbf{t}$, see \cite{SM92}, \cite[Thm.~9.7]{ORV08}
and \cite[Sec.~10]{Raf16}.)

\section{Residuated Structures}\label{residuated structures section}

\begin{defn}
\textup{An
{\em involutive (commutative) residuated lattice}, or briefly, an {\em IRL},
is an algebra
$\sbA=\langle A;\bcdw,\wedge,\vee,\neg,e\rangle$
comprising a commutative monoid $\langle A;\bcdw,e\rangle$,
a lattice $\langle A;\wedge,\vee\rangle$
and a function $\neg\colon A\mrig A$,
called an {\em involution},
such that $\sbA$ satisfies the (first order) formulas $\neg\neg x=x$ and
\begin{equation}\label{involution-fusion law}
x\bcdw y\leqslant z\;\Longleftrightarrow\;\neg z\bcdw y\leqslant\neg x,
\end{equation}
cf.\ \cite{GJKO07}.\footnote{\,The signature in \cite{GJKO07} is slightly different, but the definable terms are not affected.}
Here, $\leqslant$ denotes the lattice order (i.e., $x\leqslant y$ abbreviates $x\wedge y=x$)
and $\neg$ binds more strongly than any other operation; we refer to $\bcdw$
as {\em fusion}.}
\end{defn}
Setting $y=e$ in (\ref{involution-fusion law}), we see that $\neg$ is antitone.  In fact, De Morgan's laws for $\neg,\wedge,\vee$ hold, so
$\neg$ is an anti-automorphism of $\langle A;\wedge,\vee\rangle$.
If we define
\[
x\rig y\seteq\neg(x\bcdw\neg y)
\textup{
 \ and \ } f\seteq\neg e,
\]
then, as is well known, every IRL satisfies
\begin{align}
& x\bcdw y\leqslant z\;\Longleftrightarrow\;y\leqslant x\rig z \quad\textup{(the law of residuation),}\label{residuation}\\
& \neg x= x\rig f \textup{ \ and \ } x\rig y=\neg y\rig\neg x \textup{ \ and \ } x\bcdw y=\neg(x\rig\neg y).\label{neg properties}
\end{align}

\begin{defn}
\textup{A {\em (commutative) residuated lattice\/}---or
an {\em RL\/}---is an algebra
$\sbA=\langle A;\bcdw,\rig,\wedge,\vee,e\rangle$
comprising a commutative monoid $\langle A;\bcdw,e\rangle$, a lattice $\langle A;\wedge,\vee\rangle$ and
a binary operation $\rig$, called \emph{residuation}, where
$\sbA$ satisfies (\ref{residuation}).
}
\end{defn}
Thus, up to term equivalence, every IRL has a reduct that is an RL.  Conversely,
every RL can be embedded into (the RL-reduct of) an IRL; see \cite{GR04} and the
antecedents cited there.  Every RL satisfies the following well known formulas.  
Here and subsequently, $x\leftrightarrow y$ abbreviates $\textup{$(x\rig y)\wedge(y\rig x)$}$. \enlargethispage{5pt}
\begin{align}
& x\bcdw (x\rig y)\leqslant y
\textup{ \ and \ }  x\leqslant (x\rig y)\rig y
\label{x y law}\\
& (x\bcdw y)\rig z=y\rig (x\rig z)=x\rig (y\rig z) \label{permutation}\\
& (x\rig y)\bcdw(y\rig z)\leqslant x\rig z \label{transitivity}\\
& x\bcdw(y\vee z)=(x\bcdw y)\vee(x\bcdw z)\label{fusion distributivity}\\
& x\leqslant y\;\Longrightarrow\;\left(
                           (x\bcdw z\leqslant y\bcdw z) \;\;\&
                           \;\;
                           (z\rig x\leqslant z\rig y)\;\;\&
                           \;\;
                           (y\rig z\leqslant x\rig z)\right)
\label{isotone}
\\
& x\leqslant y\;\Longleftrightarrow\;e\leqslant x\rig y \label{t order}\\
& x=y\;\Longleftrightarrow\;e\leqslant x\leftrightarrow y \label{t reg}\\
& e\leqslant x\rig x \textup{ \ and \ } e\rig x=x \label{t laws}\\
& e\leqslant x\;\Longleftrightarrow \; x\rig x\leqslant x. \label{t eq}
\end{align}

By (\ref{t reg}), an RL $\sbA$ is \emph{nontrivial} (i.e., $\left| A\right|>1$)
iff $e$ is not its least element, iff $e$ has a strict lower bound.  A class of
algebras is said to be \emph{nontrivial} if it has a nontrivial member.

Another
consequence of (\ref{t reg}) is that a non-injective homomorphism $h$ between RLs
must satisfy $h(c)=e$ for some $c<e$.  (Choose $c=e\wedge(a\leftrightarrow b)$,
where $h(a)=h(b)$ but $a\neq b$.)

In an RL,
we define $x^0\seteq e$ and $x^{n+1}\seteq x^n\bcdw x$ \,for $n\in\omega$.

\begin{lem}\label{bounds}
If a (possibly involutive) RL\/ $\sbA$ has a least element\/ $\bot$\textup{,} then\/
$\top\seteq\bot\rig\bot$ is its greatest element
and, for all\/
$a\in A$\textup{,}
\[
a\bcdw\bot=\bot=\top\rig\bot
\text{ \ and\/ \ }
\bot\rig a=\top=
a\rig\top=\top^2.
\]
In particular, $\{\bot,\top\}$ is a subalgebra of the\/ $\bcdw,\rig,\wedge,\vee\,(,\neg)$
reduct of\/ $\sbA$\textup{.}
\end{lem}
\begin{proof}
See \cite[Prop.~5.1]{OR07}, for instance.  (We infer $\top=\top^2$
from (\ref{isotone}), as $e\leqslant\top$.
The lattice anti-automorphism $\neg$, if present, clearly switches $\bot$ and $\top$.)
\end{proof}

If we say that $\bot,\top$ are {\em extrema\/} of an RL $\sbA$, we mean that $\bot\leqslant a\leqslant\top$
for all $a\in A$.  An RL with extrema is said to be {\em bounded}.  In that case, its extrema need not be
{\em distinguished\/} elements, so they are not always retained in subalgebras.
The next lemma is a straightforward consequence
of (\ref{residuation}).
\begin{lem}\label{rigorously compact}
The following conditions on a bounded IRL $\sbA$\textup{,} with extrema\/ $\bot,\top$\textup{,}
are equivalent.
\begin{enumerate}
\item
$\top\bcdw a=\top$ whenever\/ $\bot\neq a\in A$\textup{.}
\item
$a\rig\bot=\bot$ whenever\/ $\bot\neq a\in A$\textup{.}
\item
$\top\rig b=\bot$ whenever\/ $\top\neq b\in A$\textup{.}
\end{enumerate}
\end{lem}
\begin{defn}\label{rigorously compact definition}
\textup{Following Meyer \cite{Mey86}, we say that an
IRL is
{\em rigorously compact\/} if it is bounded and satisfies
the equivalent conditions of Lemma~\ref{rigorously compact}.}
\end{defn}

\begin{lem}\label{3 conditions}
Let\/ $\sbA$ be an IRL, with $a\in A$\textup{.}  Then\/
\begin{equation*}
\text{$e\leqslant a=a^2$ \ iff \ $a\bcdw\neg a=\neg a$ \ iff \ $a=a\rig a$\textup{.}}
\end{equation*}
\end{lem}
\begin{proof}
The second and third conditions are equivalent, by the definition of $\rig$ and involution properties.
Also, $a^2\leqslant a$ and $a\bcdw\neg a\leqslant \neg a$ are equivalent, by (\ref{involution-fusion law}).
From
$e\leqslant a$ and (\ref{isotone}) we infer
$\neg a=e\bcdw\neg a\leqslant a\bcdw\neg a$.  Conversely,
$a\rig a\leqslant a$ and (\ref{t laws}) yield $e\leqslant a$,
and therefore $a\leqslant a^2$.
\end{proof}

The class of all RLs and that of all IRLs
are finitely axiomatized varieties.  They are arithmetical
(i.e., congruence distributive and congruence permutable) and have
the congruence extension property (CEP).
These facts can be found, for instance, in \cite[Sections~2.2 and 3.6]{GJKO07}.

\section{Square-Increasing IRLs}\label{square increasing icrls section}

An RL is said to be {\em square-increasing\/} if it satisfies $x\leqslant x^2$.
Every square-increasing
RL can be embedded into a square-increasing IRL \cite{Mey73}.
Moreover, Slaney \cite{Sla16} has shown that if two square-increasing IRLs have the same RL-reduct, then they are equal.
The following formulas are valid in all square-increasing IRLs (and not in all IRLs):
\begin{eqnarray}
&& x\wedge y\leqslant x\bcdw y
\label{square increasing cor} \\
&& (x\leqslant e\;\;\&\;\;y\leqslant e)\;\Longrightarrow \; x\bcdw y=x\wedge y\label{square increasing cor2}\\
&& e\leqslant x\vee\neg x. \label{excluded middle}
\end{eqnarray}

The lemma below generalizes another result of Slaney \cite[T36, p.\,491]{Sla85} (where only the case $a=f$ was discussed, and $\sbA$
satisfied an extra postulate).

\begin{lem}\label{cube}
Let\/ $\sbA$ be a square-increasing IRL, with\/ $f\leqslant a\in A$\textup{.}  Then\/ $a^3=a^2$\textup{.}  In particular, $f^3=f^2$\textup{.}
\end{lem}
\begin{proof}
As $f\leqslant a$, we have $\neg a=a\rig f\leqslant a\rig a$, by (\ref{neg properties}) and (\ref{isotone}), so
\begin{equation}\label{theophilus}
a\rig\neg a\leqslant a\rig(a\rig a)=a^2\rig a,
\end{equation}
by (\ref{isotone}) and (\ref{permutation}).
By the square-increasing law, (\ref{theophilus}), (\ref{isotone})
and (\ref{transitivity}),
\[
a\rig \neg a\leqslant (a\rig\neg a)^2\leqslant
(a^2\rig a)\bcdw(a\rig\neg a)
\leqslant a^2\rig\neg a.
\]
Thus, $\neg(a^2\rig\neg a)\leqslant\neg(a\rig\neg a)$, i.e., $a^2\bcdw a\leqslant a\bcdw a$ (see (\ref{neg properties})), i.e., $a^3\leqslant a^2$.
The reverse inequality follows from the square-increasing law and
(\ref{isotone}).
\end{proof}

The first assertion of the next theorem has unpublished antecedents in the work of relevance logicians.  A corresponding result for `relevant algebras' is reported
in \cite[Prop.~5]{Swi99},
but the claim and proof below are simpler.

\begin{thm}\label{bounded}
Every finitely generated square-increasing IRL\/ $\sbA$ is bounded.
More precisely,
let\/ $\{a_1,\dots,a_n\}$ be a finite set of generators for $\sbA$\textup{,} with
\[
c=e\vee f\vee\,\textup{{\small $\bigvee$}}_{\!i\,\leq\, n}(a_i\vee\neg a_i), \text{ \ and\/ \ } b=c^2.
\]
Then\/ $\neg b\leqslant a\leqslant b$ for all\/ $a\in A$\textup{.}
\end{thm}
\begin{proof}
By De Morgan's laws, every element of $\sbA$ has the form
$\varphi^\sbA(a_1,\dots,a_n)$ for some term $\varphi(x_1,\dots,x_n)$ in the language $\bcdw,\wedge,\neg,e$.  The proof of the present theorem
is by induction on the
complexity $\#\varphi$ of $\varphi$.  We shall write $\ov{x}$ and $\ov{a}$ for the respective sequences $x_1,\dots,x_n$ and $a_1,\dots,a_n$.

For the case $\#\varphi\leq 1$, note that
$e,a_1,\dots,a_n\leqslant c\leqslant b$,
by the square-increasing law.  Likewise, $f,\neg a_1,\dots,\neg a_n\leqslant c\leqslant b$, so by involution properties,
$\neg b\leqslant e,a_1,\dots,a_n$.
Now suppose $\#\varphi>1$ and that
$\neg b\leqslant\psi^\sbA(\ov{a})\leqslant b$
for all terms $\psi$ with $\#\psi<\#\varphi$.  The desired result, viz.\
\[
\neg b\leqslant\varphi^\sbA(\ov{a})\leqslant b,
\]
follows from the induction hypothesis and basic properties of IRLs if $\varphi$ has the form $\neg\psi(\ov{x})$ or $\psi_1(\ov{x})\wedge\psi_2(\ov{x})$.
We may therefore assume that $\varphi$ is $\psi_1(\ov{x})\bcdw\psi_2(\ov{x})$ for some less complex terms $\psi_1(\ov{x}),\psi_2(\ov{x})$.

By the induction hypothesis and (\ref{isotone}), $(\neg b)^2\leqslant\varphi^\sbA(\ov{a})\leqslant b^2$.
As $\neg b\leqslant e$, we have
$(\neg b)^2=\neg b$, by (\ref{square increasing cor2}).  And since $f\leqslant c$, Lemma~\ref{cube} gives
$c^3=c^2$,
so $b^2=c^4=c^2=b$.  Therefore, $\neg b\leqslant\varphi^\sbA(\ov{a})\leqslant b$,
as required.
\end{proof}

In a square-increasing IRL, the smallest subalgebra $\sbB$ (generated by $\emptyset$) has top element
$(e\vee f)^2=f^2\vee e$ (by Theorem~\ref{bounded} and (\ref{fusion distributivity})).  This is a lower
bound of $f\rig f^2$ (by (\ref{residuation}) and Lemma~\ref{cube}), so $f^2\vee e=f\rig f^2$.  That the extrema of
$\sbB$ can be
expressed without using $\wedge,\vee$ is implicit in \cite[p.\,309]{Mey86}.  Note also
that $e\leftrightarrow f=f\wedge\neg(f^2)$ is the least element of $\sbB$.

An element $a$ of an [I]RL $\sbA$ is said to be {\em idempotent\/} if $a^2=a$.  We say that $\sbA$ is
{\em idempotent\/} if all of its elements are.
In the next result, the key implication is
(\ref{f<=e})$\,\Rig\,$(\ref{a is idempotent}).  A logical analogue of
(\ref{f<=e})$\,\Leftrightarrow\,$(\ref{a is idempotent}) is stated without proof in \cite[p.\,309]{Mey86}.

\begin{thm}\label{idempotence f and e}
In a square-increasing IRL $\sbA$\textup{,} the following are equivalent.
\begin{enumerate}
\item\label{f^2=f}
$f^2=f$\textup{.}
\item\label{f<=e}
$f\leqslant e$\textup{.}
\item\label{a is idempotent}
$\sbA$ is idempotent.
\end{enumerate}
Consequently, a square-increasing non-idempotent IRL has no idempotent subalgebra (and in particular, no trivial subalgebra).
\end{thm}
\begin{proof}
In any IRL, (\ref{f^2=f})$\,\Rig\,$(\ref{f<=e}) instantiates (\ref{involution-fusion law}) (as $\neg f=e$), and
(\ref{a is idempotent})$\,\Rig\,$(\ref{f^2=f}) is trivial.

(\ref{f<=e})$\,\Rig\,$(\ref{a is idempotent}):
Suppose $f\leqslant e$, and let $a\in A$.  It suffices to show that $a^2\leqslant a$, or equivalently (by (\ref{involution-fusion law})),
that $a\bcdw\neg a\leqslant\neg a$.  Now, by the square-increasing law, (\ref{isotone}), the associativity of fusion, (\ref{neg properties}) and (\ref{x y law}),
\[
a\bcdw\neg a\,\leqslant\,a\bcdw(\neg a)^2\,=\,(a\bcdw(a\rig f))\bcdw\neg a\,\leqslant\,f\bcdw\neg a\,\leqslant\,e\bcdw\neg a\,=\,\neg a. \qedhere
\]
\end{proof}

In a partially ordered set, we denote by $[a)$ the set of all upper bounds of an element $a$ (including $a$ itself), and by $(a]$ the set of all lower
bounds.

A {\em deductive filter\/} of a (possibly involutive) RL $\sbA$ is a lattice filter $G$ of $\langle A;\wedge,\vee\rangle$ that is also a submonoid
of $\langle A;\bcdw,e\rangle$.
Thus, $[e)$ is the smallest deductive filter of $\sbA$.  The lattice of deductive filters of $\sbA$ and the congruence lattice
${\boldsymbol{\mathit{Con}}}\,\sbA$ of $\sbA$ are isomorphic.  The isomorphism and its inverse are given by
\begin{eqnarray*}
& G\,\mapsto\,\leibniz
G
\seteq
\{\langle a,b\rangle\in A^2\colon a\rig b,\,b\rig a\in G\};\\
& \theta\,\mapsto\,\{a\in A\colon \langle a\wedge e, e\rangle\in\theta\}.
\end{eqnarray*}
For a deductive filter $G$ of $\sbA$ and $a,b\in A$, we often abbreviate $\sbA/\leibniz G$ as $\sbA/G$, and $a/\leibniz G$ as $a/G$, noting that
\begin{equation}\label{factor order}
a\rig b\in G \textup{ \ iff \ } a/G\leqslant b/G \textup{ \,in }\sbA/G.
\end{equation}
In the square-increasing case, the deductive filters of $\sbA$ are just the lattice filters of $\langle A;\wedge,\vee\rangle$ that contain $e$, by
(\ref{square increasing cor}).  This yields the following lemma.
\begin{lem}\label{ded filter}
In a square-increasing IRL\/ $\sbA$\textup{,}
\begin{enumerate}
\item\label{ded filter 1}
if\/ $e\geqslant b\in A$\textup{,} then\/ $[b)$ is a deductive filter of\/ $\sbA$\textup{,} e.g.,

\smallskip

\item\label{ded filter 2}
$[\neg(f^2))$
is a deductive filter of\/ $\sbA$\textup{.}
\end{enumerate}
\end{lem}
\noindent
Here, (\ref{ded filter 2}) follows from (\ref{ded filter 1}), because $e\geqslant\neg(f^2)$ follows from $f\leqslant f^2$.
\begin{thm}\label{Sugihara factor}
Let\/ $G$ be a deductive filter of a square-increasing IRL $\sbA$\textup{.}
Then\/
$\sbA/G$ is idempotent iff\/ $\neg(f^2)\in G$\textup{.}
In particular, $\sbA/[\neg(f^2))$ is idempotent.
\end{thm}
\begin{proof}
$\sbA/G$ is idempotent iff $f/G \leqslant e/G$ (by Theorem~\ref{idempotence f and e}),
iff $f\rig e\in G$ (by (\ref{factor order})),
iff $\neg(f^2)\in G$ (as $\neg(f^2)=\neg(f\bcdw\neg e)=f\rig e$).
\end{proof}

We say that a square-increasing IRL is \emph{anti-idempotent\/} if it satisfies $x\leqslant f^2$ (or equivalently,
$\neg(f^2)\leqslant x$).
This terminology is justified by the corollary below.

\begin{cor}\label{sug cor}
Let\/ $\mathsf{K}$ be a variety of square-increasing IRLs.  Then\/ $\mathsf{K}$ has no nontrivial idempotent member
iff it satisfies\/ $x\leqslant f^2$\textup{.}
\end{cor}
\begin{proof}
($\Rig$):  As $\mathsf{K}$ is homomorphically closed but lacks nontrivial idempotent members, Theorem~\ref{Sugihara factor}
shows that the deductive filter $[\neg(f^2))$ of any $\sbA\in\mathsf{K}$ coincides with $A$, i.e., $\mathsf{K}$ satisfies
$\neg(f^2)\leqslant x$.

($\Leftarrow$):  If $\sbA\in\mathsf{K}$ is idempotent, then
$f^2=f\leqslant e=\neg f=\neg(f^2)$, by Theorem~\ref{idempotence f and e},
so by assumption, $\sbA$ is trivial.
\end{proof}

Recall that an algebra $\sbA$ is {\em subdirectly irreducible\/} (SI) if its identity relation
$\textup{id}_A=\{\langle a,a\rangle\colon a\in A\}$ is completely
meet-irreducible in its congruence lattice; see for instance \cite[Thm.~3.23]{Ber12}.
If $\textup{id}_A$ is merely meet-irreducible in
${\boldsymbol{\mathit{Con}}}\,\sbA$, then $\sbA$ is said to be {\em finitely subdirectly irreducible\/}
(FSI), whereas $\sbA$ is {\em simple\/} if
\,$\textup{$\left|{{\mathit{Con}}}\,\sbA\right|=2$}$.  (Thus, trivial algebras are FSI, but are neither SI nor simple.)

By Birkhoff's Subdirect Decomposition Theorem \cite[Thm.~3.24]{Ber12},
every algebra is isomorphic to a subdirect product of SI homomorphic images of itself.
(Even a trivial algebra is a copy of the direct product of an empty family.)
Also, every algebra embeds into an ultraproduct of finitely generated subalgebras of itself \cite[Thm.~V.2.14]{BS81}.
Consequently, every variety is generated---and thus determined---by its SI finitely generated members, so we need to
understand these algebras in the present context.
The following result is well known; see \cite[Cor.~14]{GOR08} and \cite[Thm.~2.4]{OR07}, for instance.  Here and subsequently, an RL $\sbA$
is said to be \emph{distributive}
if its reduct $\langle A;\wedge,\vee\rangle$ is a distributive
lattice.

\begin{lem}\label{fsi si simple}
Let\/ $\sbA$ be
a (possibly involutive) RL.
\begin{enumerate}
\item\label{fsi}
$\sbA$ is FSI iff\/ $e$ is join-irreducible in\/ $\langle A;\wedge,\vee\rangle$\textup{.}
Therefore, subalgebras
and ultraproducts
of FSI \,[I]RLs are FSI.

\smallskip

\item\label{prime}
When\/ $\sbA$
is distributive, it
is FSI iff\/ $e$ is join-prime\/
\textup{(}i.e.,
whenever $a,b\in A$ with $e\leqslant a\vee b$\textup{,} then $e\leqslant a$ or $e\leqslant b$\textup{).}

\smallskip

\item\label{si}
If there is a largest element strictly below $e$\textup{,} then $\sbA$ is SI.  \,The converse holds if
$\sbA$ is square-increasing.

\smallskip

\item\label{simple}
If\/ $e$ has just one strict lower bound, then\/ $\sbA$ is simple.  The converse holds when\/ $\sbA$ is square-increasing.
\end{enumerate}
\end{lem}

An [I]RL is said to be {\em semilinear\/} if it is
isomorphic to a subdirect product of totally ordered algebras; it is
{\em integral\/} if $e$ is its greatest element, in which case it satisfies $e=x\rig x=x\rig e$.
A {\em Brouwerian algebra\/} is an integral idempotent RL, i.e., an RL satisfying $x\bcdw y=x\wedge y$.  Such an algebra is
determined by its lattice reduct, and
is distributive, by (\ref{fusion distributivity}).  The variety
of \emph{relative Stone algebras} comprises the semilinear Brouwerian algebras; it is generated by the
Brouwerian algebra on the chain of non-negative integers.

\section{De Morgan Monoids}\label{dmm section}

\begin{defn}\label{de morgan monoid definition}
\textup{A {\em De Morgan monoid\/} is a distributive square-increasing IRL.\footnote{\,But see the first paragraph of Section~\ref{relevant algebras}.}
\,The variety of De Morgan monoids shall be denoted by $\mathsf{DMM}$.}
\end{defn}

The following lemma is well known and should be contrasted with the previous section's concluding remarks about involutionless algebras.

\begin{lem}\label{integral implies boolean}
A De Morgan monoid is integral iff it is a Boolean algebra (in which the operation\/ $\wedge$ is duplicated by fusion).
\end{lem}
\begin{proof}
Sufficiency is clear.  Conversely, by (\ref{excluded middle}) and De Morgan's laws,
the fusionless
reduct of an integral De Morgan monoid is a
complemented (boun\-ded)
distributive lattice, i.e., a Boolean algebra, and
$\bcdw$ is $\wedge$, by (\ref{square increasing cor2}).
\end{proof}

An algebra is said to be $n$--\emph{generated} (where $n$ is a cardinal) if it has a generating subset with at most $n$ elements.
Thus, an IRL is $0$--generated iff it has no proper subalgebra.

Infinite
$1$--generated De Morgan monoids exist.  Indeed, the integer powers of $2$, with the usual order and ordinary multiplication as fusion,
can be extended to an algebra of this kind.
The larger varieties of distributive and of square-increasing IRLs each have infinite
$0$--generated members as well \cite{Sla93}, but
Slaney
proved that the free $0$--generated De Morgan monoid has just $3088$ elements \cite{Sla85}.
His arguments show that, up to isomorphism,
only eight $0$--generated De Morgan monoids are FSI; they are exhibited in \cite{Sla89}.  As the seven nontrivial $0$--generated
FSI De Morgan monoids are finite, they are just the $0$--generated SI \,De Morgan monoids.

A theorem of Urquhart \cite{Urq84} implies that the equational theory of $\mathsf{DMM}$ is undecidable, whereas results in
\cite{Bra90,Koz,Mey01} show that the respective varieties of distributive and of square-increasing IRLs are generated by their finite
members, whence their equational theories are decidable (although, in the square-increasing case, no primitive
recursive decision procedure exists \cite{Urq99b}).

Recall that a {\em quasivariety\/} is a class of similar algebras closed under isomorphic images, subalgebras, direct products
and ultraproducts.  Equivalently, it is the model class of some set of pure {\em quasi-equations\/}
\[
(\al_1=\be_1\;\;\&\;\;\dots\;\;\&\;\;\al_n=\be_n)\;\Longrightarrow\;\al=\be
\]
in an algebraic signature.
Here $n\in\omega$, i.e., quasi-equations have finite length
and encompass equations.  Although a quasivariety need not be homomorphically closed (i.e., it need not be a variety), it must
contain a trivial algebra, viz.\ the direct product of its empty subfamily.

\subsection{Relevance Logic and De Morgan Monoids} \label{relevance logic}

$\mbox{}$

\smallskip

For present purposes, a {\em logic\/} is a substitution-invariant finitary consequence relation $\,\vdash$ over sentential formulas
in an algebraic signature, cf.\ \cite{BP89,Cze01,Fon16,FJP03}.
The general connections between
residuated structures and {\em substructural logics\/} are explained in \cite{GJKO07}.  In the case of De Morgan monoids, the connection is with
the older family of {\em relevance logics\/} (a.k.a.\ relevant logics).  The monographs and survey articles on this subject include \cite{AB75,ABD92,Bra03,DR01,Mar06,MM01,Rea88,Res01,RMPB82}.
The correspondence is as follows.

For each subquasivariety $\mathsf{K}$ of $\mathsf{DMM}$, there is a logic $\,\vdash^\mathsf{K}$ with the same signature, defined thus:
for any set $\Gamma\cup\{\al\}$ of formulas, $\Gamma\vdash^\mathsf{K}\al$ iff there exist $n\in\omega$ and $\ga_1,\dots,\ga_n\in\Gamma$ such that every algebra in $\mathsf{K}$ satisfies
\[
e\leqslant \ga_1\wedge\,\dots\,\wedge\ga_n\;\Longrightarrow\;e\leqslant\al.
\]
The elements (also called the {\em derivable rules\/}) of $\,\vdash^\mathsf{K}$ are just the pairs $\Gamma/\al$ for which this is true.
In particular, the {\em theorems\/} of $\,\vdash^\mathsf{K}$ (i.e., the formulas $\al$ for which $\emptyset\vdash^\mathsf{K}\!\al$) are
just the IRL
terms that take values in $[e)$ whenever their variables are interpreted in any member of $\mathsf{K}$.

Because $\mathsf{DMM}$ satisfies (\ref{t reg}),
the logic
$\,\vdash^\mathsf{K}$ is {\em algebraizable\/} in the sense of \cite{BP89}, with $\mathsf{K}$ as its unique equivalent quasivariety.
The map $\mathsf{K}\mapsto\,\,\vdash^\mathsf{K}$ is a lattice anti-isomorphism from the subquasivarieties of $\mathsf{DMM}$
to the extensions of the relevance logic $\mathbf{R}^\mathbf{t}$
of \cite{AB75},
carrying the subvarieties of $\mathsf{DMM}$ onto the axiomatic extensions.
In particular, $\mathbf{R}^\mathbf{t}$ itself is algebraized by $\mathsf{DMM}$.

The relationship between $\mathbf{R}^\mathbf{t}$ and $\mathsf{DMM}$ was essentially established by Dunn \cite{Dun66} (see his contributions to \cite{AB75}, as well as \cite{MDL74}).
Strictly speaking, $\mathbf{R}^\mathbf{t}$ denotes a formal system of
axioms and inference rules (see \cite[pp.~341--343]{AB75}), not a consequence relation.  Here, however, we routinely attribute to a formal system $\mathbf{F}$ the significant properties
of its deducibility relation $\,\vdash_\mathbf{F}$.\,\footnote{\,The general theory of algebraization \cite{BP89}
applies only to consequence relations.
This is in contrast with a tradition---prevalent in relevance logic and elsewhere---of
identifying
a `logic' with its set of theorems alone, leaving its rules of derivation under-determined in the absence of further qualification.  The same tradition
privileges \emph{axiomatic} extensions.
No serious ambiguity
ensues in the case of $\mathbf{R}^\mathbf{t}$, as we can recover the whole of $\,\vdash_{\mathbf{R}^\mathbf{t}}$ from its theorems,
via the so-called {\em enthymematic deduction
theorem\/}:
$\Gamma,\al\vdash_{\mathbf{R}^\mathbf{t}}\beta \textup{ \,iff\, } \Gamma\vdash_{\mathbf{R}^\mathbf{t}}(\al\wedge\mathbf{t})\rig\beta$ (\cite{Mey73a}).}

Although relevance logic has multiple interpretations (see for instance \cite{RM73,SM94,Urq72,Urq83,Urq96}), it
was originally intended as a framework in which the so-called paradoxes of material implication could be avoided.
These include the {\em weakening axiom\/} $p\rig(q\rig p)$.  The unprovability of this postulate
in $\mathbf{R}^\mathbf{t}$ reflects the fact that De Morgan monoids need not be integral, and Lemma~\ref{integral implies boolean}
says in effect that classical propositional logic is the extension of $\mathbf{R}^\mathbf{t}$ by the weakening axiom.
Another relevance logic $\mathbf{R}$, and its connection with De Morgan monoids, will be discussed in Section~\ref{relevant algebras}.

\section{The Structure of De Morgan Monoids}\label{structure section}

In the relevance logic literature, a De Morgan monoid is said to be {\em prime\/} if it is FSI.  The reason is Lemma~\ref{fsi si simple}(\ref{prime}),
but we continue to use `FSI' here, as it makes sense for arbitrary algebras.
The next result is easy and well known, but note that it draws on
all the key properties of De Morgan monoids.

\begin{thm}\label{splitting}
Let\/ $\sbA$ be a De Morgan monoid that is FSI, with\/ $a\in A$\textup{.}  Then\/ $e\leqslant a$ or\/ $a\leqslant f$\textup{.}  Thus, $A=[e)\cup(f\,]$\textup{.}
\end{thm}
\begin{proof}
As $\sbA$ is square-increasing, $e\leqslant a\vee\neg a$, by (\ref{excluded middle}).  So, because $\sbA$ is distributive and FSI, $e\leqslant a$ or $e\leqslant\neg a$,
by Lemma~\ref{fsi si simple}(\ref{prime}).
In the latter case, $a\leqslant f$, because $\neg$ is antitone.
\end{proof}
\begin{cor}\label{subcover}
Let\/ $\sbA$ be a De Morgan monoid that is SI.  Let\/ $c$ be the largest element of\/ $\sbA$
strictly below\/ $e$\/ \,\textup{(}which exists, by Lemma~\textup{\ref{fsi si simple}(\ref{si})).}  Then $c\leqslant f$.
\end{cor}

The following result about bounded De Morgan monoids was essentially proved by Meyer \cite[Thm.~3]{Mey86},
but his argument assumes that the elements $\bot,\top$ are distinguished, or at least definable in terms of
generators.  To avoid that presupposition, we
give a simpler and more direct proof.

\begin{thm}\label{dm fsi rigorously compact}
Let\/ $\sbA$ be a bounded FSI\, De Morgan monoid.  Then\/ $\sbA$ is rigorously compact\/ \textup{(}see Definition~\textup{\ref{rigorously compact definition}).}
\end{thm}
\begin{proof}
Let $\bot\neq a\in A$, where $\bot,\top$ are the extrema of $\sbA$.
\,It suffices to show that $\top\bcdw a=\top$.
\,As $e\bcdw a\nleqslant\bot$, we have
$\top\bcdw a\not\leqslant f$, by (\ref{involution-fusion law}), so
\begin{equation}\label{top equation}
e\leqslant\top\bcdw a,
\end{equation}
by Theorem~\ref{splitting}.  Recall
that $\top^2=\top$, by Lemma~\ref{bounds}.
Therefore,
\[
\textup{$\top=\top\bcdw e\leqslant\top^2\bcdw a$\, (by (\ref{top equation})) \,$=\top\bcdw a\leqslant\top$,}
\]
whence $\top\bcdw a=\top$.
\end{proof}

\begin{cor}\label{rigorously compact generation}
If a De Morgan monoid is FSI, then its finitely generated subalgebras are rigorously compact.
\end{cor}
\begin{proof}
This follows from Lemma~\ref{fsi si simple}(\ref{fsi}) and Theorems~\ref{bounded} and \ref{dm fsi rigorously compact}.
\end{proof}

At this point, we need to recall a few concepts and results from universal algebra.
The class operator symbols $\mathbb{I}$, $\mathbb{H}$, $\mathbb{S}$, $\mathbb{P}$, $\mathbb{P}_\mathbb{S}$ and $\mathbb{P}_\mathbb{U}$
stand, respectively, for closure under isomorphic and homomorphic images, subalgebras, direct and subdirect products, and ultra\-products.
Also, $\mathbb{V}$ and $\mathbb{Q}$ denote varietal and quasivarietal closure, i.e., $\mathbb{V}=\mathbb{HSP}$ and
$\mathbb{Q}=\mathbb{ISPP}_\mathbb{U}$.  We abbreviate $\mathbb{V}(\{\sbA\})$ as $\Vop(\Alg{A})$, etc.

A variety $\mathsf{K}$ is said to be
\emph{finitely generated} if $\mathsf{K}=\mathbb{V}(\sbA)$ for some finite algebra $\sbA$ (or equivalently,
$\mathsf{K}=\mathbb{V}(\mathsf{L})$ for some finite set $\mathsf{L}$ of finite algebras).  Every finitely generated variety is \emph{locally finite},
i.e., its finitely generated members are finite algebras
\cite[Thm.~II.10.16]{BS81}.

Recall that $\mathbb{P}_\mathbb{U}(\mathsf{L})\subseteq\mathbb{I}(\mathsf{L})$ for any finite set $\mathsf{L}$ of
finite similar algebras.
Given a class $\mathsf{L}$ of algebras, let us denote by $\mathsf{L}_\textup{FSI}$ the class of all FSI
members of $\mathsf{L}$.

{\em J\'{o}nsson's Theorem\/} \cite{Jon67,Jon95} asserts that,
if $\mathsf{L}$ is contained in a
congruence distributive variety, then
$\mathbb{V}(\mathsf{L})_\textup{FSI}\subseteq \mathbb{HSP}_\mathbb{U}(\mathsf{L})$.
In particular, if $\mathsf{L}$ consists of finitely many finite similar algebras and $\mathbb{V}(\mathsf{L})$ is congruence distributive, then
$\mathbb{V}(\mathsf{L})_\textup{FSI}\subseteq\mathbb{HS}(\mathsf{L})$.

As RLs are congruence distributive, J\'{o}nsson's Theorem shows that, whenever $\mathsf{L}$ consists of totally ordered [I]RLs, then so does
$\mathbb{V}(\mathsf{L})_\textup{FSI}$, whence $\mathbb{V}(\mathsf{L})$ consists of semilinear algebras.  Indeed, since
total order is expressible by a universal positive first order sentence, it persists
under the operators $\mathbb{H}$, $\mathbb{S}$ and $\mathbb{P}_\mathbb{U}$.

\begin{defn}\label{sug def}
\textup{A \emph{Sugihara monoid} is an idempotent De Morgan monoid, i.e.,
an idempotent distributive IRL.}
\end{defn}

The variety $\mathsf{SM}$ of Sugihara monoids is well understood,
largely because of Dunn's
contributions
to \cite{AB75}; see \cite{Dun70} also.
It is locally finite, but not finitely generated.
In fact, $\mathsf{SM}$ is the smallest variety containing the unique Sugihara monoid
\[
\sbS^*=\langle\{a:0\neq a\in\mathbb{Z}\};\bcdw,\wedge,\vee,-,1\rangle
\]
on the set of all nonzero integers such that the lattice order is the usual total
order, the involution $-$ is the usual additive inversion, and
the term function of $\left|x\right|\seteq x\rig x$ is the natural absolute value function.  In this algebra,
\[
a\bcdw b \; = \; \left\{ \begin{array}{ll}
                           \textup{the element of $\{a,b\}$ with the greater absolute value, \,if $\left|a\right| \neq \left|b\right|$;}\\
                           a\wedge b  \textup{ \,if $\left|a\right| = \left|b\right|$,}
                                               \end{array}
                   \right.
\]
and the residual operation $\rig$ is given by
\begin{equation*}
a \rig b \; = \; \left\{ \begin{array}{ll}
                           (-a) \vee b   & \mbox{if\, $a \leqslant b$;} \\
                           (-a) \wedge b & \mbox{if\, $a \not\leqslant b$.}
                           \end{array}
                   \right.
\end{equation*}
Note that $e=1$ and $f=-1$ in $\sbS^*$.  The remark before Definition~\ref{sug def} yields:
\begin{lem}\label{sug fsi}
Every FSI Sugihara monoid is totally ordered.  In particular, Sugihara monoids are semilinear.
\end{lem}

\begin{defn}\label{odd def}
\textup{An IRL $\sbA$ is said to be {\em odd\/} if $f=e$ in $\sbA$.}
\end{defn}

\begin{thm}\label{odd cor}
Every odd De Morgan monoid is a Sugihara monoid.
\end{thm}
\begin{proof}
By Theorem~\ref{idempotence f and e}, every square-increasing odd IRL is idempotent.
\end{proof}

In the Sugihara monoid $\sbS=\langle\mathbb{Z};\bcdw,\wedge,\vee,-,0\rangle$
on the set of {\em all\/} integers, the operations are defined like those of $\sbS^*$, except that
$0$ takes over from $1$ as the neutral element for $\bcdw$.  Both $e$
and $f$ are $0$ in $\sbS$, so $\sbS$ is odd.
It follows from Theorem~\ref{odd cor} and Dunn's results in \cite{AB75,Dun70} that the variety of all odd Sugihara monoids is the
smallest quasivariety containing $\sbS$, and that {\sf SM} is the smallest quasivariety containing both $\sbS^*$ and $\sbS$.

For each positive integer $n$, let $\sbS_{2n}$ denote the subalgebra of $\sbS^*$ with universe $\{-n,\dots,-1,1,\dots,n\}$ and,
for $n\in\omega$, let $\sbS_{2n+1}$ be the subalgebra of $\sbS$ with universe $\{-n,\dots,-1,0,1,\dots,n\}$.
Note that $\sbS_2$ is a Boolean algebra.
The results cited above yield:
\begin{thm}\label{sug thm}
Up to isomorphism,
the algebras\/ $\sbS_n$ \textup{($1<n\in\omega$)} are precisely the finitely generated SI\, Sugihara monoids, whence the algebras\/ $\sbS_{2n+1}$ $\textup{($0<n\in\omega$)}$
are just the finitely generated SI odd Sugihara monoids.
\end{thm}
We cannot embed $\sbS$ (nor even $\sbS_{2n+1}$) into $\sbS^*$, owing to the involution.  Nevertheless, $\sbS$
is a homomorphic image of $\sbS^*$, and $\sbS_{2n+1}$ is a homomorphic image of $\sbS_{2n+2}$, for all $n\in\omega$.  In each
case, the kernel of the homomorphism identifies $-1$ with $1$; it identifies no other pair of distinct elements.  Also,
$\sbS_{2n-1}$ is a homomorphic image of $\sbS_{2n+1}$ if $n>0$; in this case the kernel collapses $-1,0,1$ to a point,
while isolating all other elements.  Thus, $\sbS_3$ is a homomorphic image of $\sbS_n$ for all $n\geq 3$.
In particular, every nontrivial variety of Sugihara monoids includes $\sbS_2$ or $\sbS_3$.
\begin{cor}\label{osm varieties}
The lattice of varieties of odd Sugihara monoids is the following chain of order type\/ $\omega+1$\textup{\,:}
\[
\mathbb{V}(\sbS_1)\subsetneq\mathbb{V}(\sbS_3)\subsetneq\mathbb{V}(\sbS_5)\subsetneq\,\dots\,\subsetneq\mathbb{V}(\sbS_{2n+1})\subsetneq\,\dots\,\subsetneq\mathbb{V}(\sbS).
\]
\end{cor}
\begin{proof}
See \cite[Sec.~29.4]{AB75} or \cite[Fact~7.6]{GR12}.
\end{proof}
Odd Sugihara monoids are categorically equivalent to relative Stone algebras \cite[Thm.~5.8]{GR12}.  The equivalence sends an odd Sugihara monoid to the set of lower bounds
of its neutral element $e$, redefining residuation as $(x\rig y)\wedge e$ and restricting the other
RL-operations, as well as all morphisms.  An analogous but more complex
result for arbitrary Sugihara monoids is proved in \cite[Thm.~10.5]{GR15} and refined in \cite[Thm.~2.24]{FG}.
The subvariety lattice of $\mathsf{SM}$ is fully described in \cite{MM12}.
Every quasivariety of odd Sugihara monoids is a variety \cite[Thm.~7.3]{GR12}.  (For a stronger result, see \cite[Thm.~9.4]{OR07}.)

As the structure of Sugihara monoids is very transparent, we concentrate now on De Morgan
monoids that are {\em not\/} idempotent.

\begin{lem}\label{idempotents above f}
Let\/ $\sbA$ be a non-idempotent FSI De Morgan monoid, and let $a$ be an idempotent element of $\sbA$\textup{.}  If\/ $a\geqslant f$\textup{,} then\/ $a>e$\textup{.}
In particular, $f^2>e$\textup{.}
\end{lem}
\begin{proof}
Suppose $a^2=a\geqslant f$.  As $\sbA$ is not idempotent, $f^2\neq f$, by Theorem~\ref{idempotence f and e}, so
$a\neq f$.  Therefore, $a\not\leqslant f$, whence $e\leqslant a$, by Theorem~\ref{splitting}.  As $f\leqslant a$, we cannot have $a=e$, by Theorem~\ref{idempotence f and e},
so $e<a$.  The last claim follows because $f^2$ is an idempotent upper bound of $f$ (by Lemma~\ref{cube}).
\end{proof}

\begin{thm}\label{Sugihara factor 2}
Let\/ $G$ be a deductive filter of a non-idempotent FSI De Morgan monoid\/ $\sbA$\textup{,} and suppose\/
$\neg(f^2)\in G$\textup{.} Then\/
$\sbA/G$ is an odd Sugihara monoid.
\end{thm}
\begin{proof}
By Theorems~\ref{Sugihara factor} and \ref{idempotence f and e},
$\sbA/G$ is idempotent and $f/G\leqslant e/G$.
By Lemma~\ref{idempotents above f},
$f^2> e$, i.e., $\neg(f^2)< f$, whence $f\in G$, i.e.,
$e\rig f\in G$.
Then
$e/G
\leqslant f/G$
(by (\ref{factor order})), so $e/G
=f/G$, as required.
\end{proof}

\begin{lem}\label{idempotents linear}
Let\/ $\sbA$ be a De Morgan monoid that is FSI, with\/ $f\leqslant a,b\in A$\textup{,} where\/ $a$ and\/ $b$ are idempotent.
Then\/ $a\leqslant b$ or\/ $b\leqslant a$\textup{.}
\end{lem}
\begin{proof}
If $\sbA$ is a Sugihara monoid,
the result follows from Lemma~\ref{sug fsi}.  We may therefore assume that $\sbA$ is not idempotent,
so $e<a,b$, by Lemma~\ref{idempotents above f}.
Then $a\bcdw\neg a=\neg a$
and $b\bcdw\neg b=\neg b$, by Lemma~\ref{3 conditions}, so
\begin{align*}
& (a\bcdw\neg b)\wedge(b\bcdw\neg a)\leqslant (a\bcdw\neg b)\bcdw(b\bcdw\neg a) \textup{ \,\,(by (\ref{square increasing cor}))}\\
& = (a\bcdw\neg a)\bcdw(b\bcdw\neg b)=\neg a\bcdw\neg b \textup{ \,\,(by the above)}\\
& =\neg a\wedge\neg b \textup{ \,\,(by (\ref{square increasing cor2}), as $\neg a,\neg b\leqslant e$).}
\end{align*}
Therefore, by De Morgan's laws,
\begin{align*}
& \,\,\neg(\neg a\wedge\neg b)\leqslant \neg((a\bcdw\neg b)\wedge(b\bcdw\neg a))\\
& =\neg(a\bcdw\neg b)\vee\neg(b\bcdw\neg a)=(a\rig b)\vee (b\rig a)
\end{align*}
and $e< a\vee b=\neg(\neg a\wedge\neg b)$, so $e< (a\rig b)\vee(b\rig a)$.  Then, since $\sbA$ is FSI,
Lemma~\ref{fsi si simple}(\ref{prime}) and (\ref{t order}) yield
$e\leqslant a\rig b$ or $e\leqslant b\rig a$, i.e., $a\leqslant b$ or $b\leqslant a$.
\end{proof}

The subalgebra of an algebra $\sbA$ generated by a subset $X$ of $A$ shall be denoted by $\boldsymbol{\mathit{Sg}}^\sbA X$.

\begin{lem}\label{a is idempotent 2}
Let\/ $\sbA$ be a De Morgan monoid that is FSI, and let $f\leqslant a\in A$\textup{,} where\/
$a\not< f^2$\textup.  Then\/ $a$ is idempotent.
\end{lem}
\begin{proof}
By Lemma~\ref{cube}, $f^2$ is idempotent, so
assume that $a\neq f^2$.
From $f\leqslant f^2$ and $a\not\leqslant f^2$, we infer $a\not\leqslant f$.  Then $e\leqslant a$, by Theorem~\ref{splitting}, so
$e,f\in[\neg a,a]\seteq\{b\in A:\neg a\leqslant b\leqslant a\}$.
Therefore,
$\neg(a^2)\leqslant x\leqslant a^2$ for all $x\in{\mathit{Sg}}^\sbA\{a\}$,
by Theorem~\ref{bounded}.  By Corollary~\ref{rigorously compact generation},
$\boldsymbol{\mathit{Sg}}^\sbA\{a\}$ is
rigorously compact.  In particular,
\begin{equation}\label{absorption}
a^2\bcdw x= a^2 \textup{ \,whenever $\neg(a^2)<x\in \mathit{Sg}^\sbA\{a\}$.}
\end{equation}
As $a\leqslant a^2$ and $a\not\leqslant f^2$, we have $a^2\not\leqslant f^2$.  But $a^2$ and $f^2$ are idempotent, by Lemma~\ref{cube}, so $f^2<a^2$, by Lemma~\ref{idempotents linear}.  Thus,
$\neg(a^2)<\neg(f^2)\in\mathit{Sg}^\sbA\{a\}$, so
\begin{equation}\label{get a life}
a^2=a^2\bcdw\neg(f^2),
\end{equation}
by (\ref{absorption}).
As $\sbA/[\neg(f^2))$ is idempotent (by Theorem~\ref{Sugihara factor}), $\neg(f^2)\leqslant a^2\rig a$, i.e.,
$a^2\bcdw\neg(f^2)\leqslant a$,
by (\ref{factor order}) and (\ref{residuation}).
Then (\ref{get a life})
gives $a^2\leqslant a$, and so $a^2=a$.
\end{proof}

\begin{thm}\label{interval subalgebra}
Let\/ $\sbA$ be a non-idempotent FSI De Morgan monoid, with\/ $f^2\leqslant a\in A$\textup{.}  Then\/ $\neg a<a$ and the interval\/
$[\neg a,a]$ is a subuniverse (i.e., the universe of a subalgebra) of\/ $\sbA$\textup{.}
In particular, $[\neg(f^2),f^2]$
is a subuniverse
of\/ $\sbA$\textup{.}
\end{thm}
\begin{proof}
In $\sbA$, we have
$\neg(f^2)\leqslant e$, as noted after Lemma~\ref{ded filter}, while
$e< f^2$, by Lemma~\ref{idempotents above f}.
Of course, $\neg a\leqslant \neg(f^2)$, so $\neg a<a$.  Thus, $[\neg a,a]$
includes $e$, and it is obviously closed under $\wedge$, $\vee$ and $\neg$.
Closure under fusion
follows from (\ref{isotone}) and the
square-increasing law,
because $a$ is idempotent (by Lemma~\ref{a is idempotent 2}).
\end{proof}

\begin{thm}\label{lollipop}
In any FSI\, De Morgan monoid,
the filter\/ $[f)$ is the union of the interval\/ $[f,f^2]$ and a chain whose least element is\/ $f^2$\textup{.}  The elements
of this chain are just the idempotent upper bounds of\/ $f$.
\end{thm}
\begin{proof}
This follows from Lemma~\ref{sug fsi} when the algebra is idempotent.  In the opposite case, the idempotent upper bounds of $f$ are exactly the upper bounds
of $f^2$ (by (\ref{isotone}) and Lemma~\ref{a is idempotent 2}), and they are comparable with all upper bounds of $f$ (by Lemmas~\ref{a is idempotent 2} and \ref{idempotents linear}).
\end{proof}

{\tiny
\thicklines
\begin{center}
\begin{picture}(250,120)(15,-39)

\put(45,19){\small{$f^2$}}
\put(48,4){\small{$f$}}
\put(50,-2){\circle*{3}}
\put(50,30){\line(0,1){15}}
\put(50,47){\line(0,1){2}}
\put(50,51){\line(0,1){2}}
\put(50,55){\line(0,1){2}}
\put(50,59){\line(0,1){2}}
\put(50,30){\circle*{3}}
\put(50,14){\circle{30}}
\put(55,51){idempotents}
\put(55,43){above $f$}
\put(45,-17){\small $[f)$}

\put(0,50){\small{\ref{lollipop}:}}


\put(245,29){\small{$f^2$}}
\put(237,7){\small{$\neg(f^2)$}}
\put(250,0){\circle*{3}}
\put(250,40){\line(0,1){15}}
\put(250,57){\line(0,1){2}}
\put(250,61){\line(0,1){2}}
\put(250,65){\line(0,1){2}}
\put(250,69){\line(0,1){2}}
\put(250,40){\circle*{3}}
\put(250,20){\circle{50}}
\put(250,0){\line(0,-1){15}}
\put(250,-17){\line(0,-1){2}}
\put(250,-21){\line(0,-1){2}}
\put(250,-25){\line(0,-1){2}}
\put(250,-29){\line(0,-1){2}}

\put(200,50){\small{\ref{lollipop 2}:}}

\end{picture}
\end{center}}

\begin{thm}\label{lollipop 2}
Any non-idempotent FSI De Morgan monoid is the union of the interval subuniverse\/ $[\neg(f^2),f^2]$ and two chains of idempotents,
$(\neg(f^2)]$ and\/ $[f^2)$\textup{.}
\end{thm}

\begin{proof}
Let $\sbA$ be a non-idempotent FSI De Morgan monoid.  Theorem~\ref{interval subalgebra} shows that $e,f\in[\neg(f^2),f^2]$ and
(with Lemma~\ref{bounds}) that $\neg(f^2)\bcdw f=\neg(f^2)$.  Note that $[f^2)$ and $(\neg(f^2)]$ are both chains
of idempotents, by Theorem~\ref{lollipop}, involution properties and (\ref{square increasing cor2}).

Suppose, with a view to contradiction, that there exists $a\in A$ such that $a\notin (\neg(f^2)]\cup[\neg(f^2),f^2]\cup [f^2)$.
By Theorem~\ref{splitting}, $e<a$ or $a<f$.  By involutional symmetry, we may assume that $e< a$.  Then $a$ is incomparable
with $f^2$ (as $a\notin [\neg(f^2),f^2]\cup [f^2)$), so $f^2\vee a>f^2$.  Also, since $f^2,a\geqslant e$, we have $f^2\bcdw a\geqslant f^2\vee a$, by (\ref{isotone}), so $f^2\bcdw a> f^2$.

Because $a>e$, we have $f\bcdw a\geqslant f$.  If $f\bcdw a\in[\neg(f^2),f^2]$, then
\[
f^2\bcdw a\leqslant (f\bcdw a)^2\leqslant f^4=f^2 \textup{ \ (by Lemma~\ref{cube}),}
\]
a contradiction.  So, by Theorem~\ref{lollipop}, $f\bcdw a$ is idempotent and $f\bcdw a>f^2$.  Then
$f\bcdw a>e,f$, and by Theorem~\ref{interval subalgebra}, $\neg(f\bcdw a)< f\bcdw a$.  This, with Theorem~\ref{bounded}, shows that
$f\bcdw a$ is the greatest element of the algebra $\sbC\seteq\boldsymbol{\mathit{Sg}}^\sbA \{f\bcdw a\}$, and $\neg(f\bcdw a)$ is the
least element of $\sbC$.  Note that $\neg(f\bcdw a)<\neg(f^2)$, as $f^2< f\bcdw a$.  Now $\sbC$ is rigorously compact, by
Corollary~\ref{rigorously compact generation}, so $\neg(f^2)\bcdw(f\bcdw a)=f\bcdw a>f^2$.  Thus, $\neg(f^2)\bcdw(f\bcdw a)\nleqslant a$,
as $f^2\nleqslant a$.

Nevertheless, as $\neg(f^2)\bcdw f=\neg(f^2)$, we have $\text{$(\neg(f^2)\bcdw f)\bcdw a=\neg(f^2)\bcdw a\leqslant a$}$, because
$\neg(f^2)\leqslant e$.  This contradicts the associativity of fusion in $\sbA$.  Therefore, $A=(\neg(f^2)]\cup[\neg(f^2),f^2]\cup[f^2)$.
\end{proof}

Recall from (\ref{square increasing cor2}) that fusion and meet coincide on the lower bounds of $e$ in any De Morgan monoid.
For the algebras in Theorem~\ref{lollipop 2}, the behaviour of fusion is further constrained as follows.

\begin{thm}\label{main structure}
Let\/ $\sbA$ be a non-idempotent FSI De Morgan monoid,
and let\/ $f\leqslant a,b\in A$\textup{.}  Then
\begin{equation*}
a \bcdw b \; = \; \left\{ \begin{array}{ll}
                           f^2   & \text{if\, $a,b \leqslant f^2$\textup{;}} \\
                           \textup{max}_{\,\leqslant}\{a,b\} & \text{otherwise.}
                           \end{array}
                   \right.
\end{equation*}
If, moreover, $a<b$ and $f^2\leqslant b$\textup{,}
then\/ $a\bcdw\neg b=\neg b=b\bcdw\neg b$ and\/ $b\bcdw\neg a=b$\textup{.}
\end{thm}
\begin{proof}
If $a,b\leqslant f^2$, then $f^2\leqslant a\bcdw b\leqslant f^4=f^2$, by (\ref{isotone}) and
Lemma~\ref{cube}, so $a\bcdw b=f^2$.
We may therefore assume (in respect of the first claim) that $a\not\leqslant f^2$ or $b\not\leqslant f^2$.
Then $a$ and $b$ are comparable, by Theorem~\ref{lollipop}.
By symmetry, we may assume that $a\leqslant b$ and
hence
that $b\not\leqslant f^2$, so $e<f^2<b=b^2$, by Theorems~\ref{interval subalgebra} and \ref{lollipop}.

If $a=b$, then
$a\bcdw b=b^2=b=\textup{max}_{\,\leqslant}\{a,b\}$, so we may assume that $a\neq b$.
Thus, $b>a\geqslant f$, and so $\neg b<\neg a\leqslant e< b$.

As $b$ is an idempotent upper bound of $e,f,a,\neg a,\neg b$,
Theorem~\ref{bounded} shows that $b$ is the greatest element of ${\boldsymbol{\mathit{Sg}}}^\sbA\{a,b\}$,
and $\neg b$ is the least element.

By Corollary~\ref{rigorously compact generation},
${\boldsymbol{\mathit{Sg}}}^\sbA\{a,b\}$ is rigorously compact.
We shall therefore have $a\bcdw b=b=\textup{max}_{\,\leqslant}\{a,b\}$, provided that $\neg b\neq a$.
This is indeed the case, as we have seen that $\neg a < b$.

Finally, suppose $a<b$ and $f^2\leqslant b$.  Again, Theorems~\ref{interval subalgebra} and
\ref{lollipop} show that $\neg b,b$ are the (idempotent) extrema of the algebra
${\boldsymbol{\mathit{Sg}}}^\sbA\{a,b\}$, whose non-extreme elements include $\neg a,a$, so the
remaining claims also follow from the rigorous compactness of ${\boldsymbol{\mathit{Sg}}}^\sbA\{a,b\}$.
\end{proof}

\begin{remk}\label{interpretation}
\emph{The foregoing results imply that, for an FSI\, De Morgan monoid $\sbA$, there are just two possibilities.}

\emph{The first is that
$f<e$, in which case, by Theorems~\ref{idempotence f and e} and \ref{splitting}
and Lemmas~\ref{fsi si simple}(\ref{si}) and \ref{sug fsi},
$\sbA$ is a totally ordered SI \,Sugihara monoid whose fusion resembles that of $\sbS^*$, because the latter
operation is definable
by universal first order sentences, and because $\sbA\in\mathbb{ISP}_\mathbb{U}(\sbS^*)$.  (See the remarks preceding
Definition~\ref{sug def} and recall
that the absolute value function on $\sbS^*$ is the term function of $x\rig x$.)  The improvement here on $\sbA\in\mathbb{HSP}_\mathbb{U}(\sbS^*)$ is due to the assumption $f<e$.  Indeed, a nontrivial congruence on any $\sbB\in\mathbb{SP}_\mathbb{U}(\sbS^*)$ must identify $f$ with $e$, because
$e$ covers $f$ in $\sbS^*$, and therefore in $\sbB$.}

\emph{The second possibility is that $\sbA$ is the `rigorous extension' of its anti-idempotent subalgebra (on $[\neg(f^2),f^2]$) by
an (idempotent) totally ordered odd Sugihara monoid.  More precisely, in this case, if $\theta=\leibniz\,[\neg(f^2))$, then
$\sbA/\theta$
is a totally ordered odd Sugihara monoid (and is therefore determined by its $e,\leqslant$ reduct),
while
$[\neg(f^2),f^2]$ is the congruence class $e/\theta$
and
no two distinct non-elements of $[\neg(f^2),f^2]$ are identified by $\theta$ (an easy consequence of Theorem~\ref{main structure}).
Thus, when $\neg(f^2)$ and $f^2$ are identified in $(\neg (f^2)]\cup[f^2)$, we obtain a copy of
$\langle A/\theta;\leqslant\rangle$.  Both $\sbA/\theta$ and the algebra on $[\neg(f^2),f^2]$ are FSI, by
Lemma~\ref{fsi si simple}(\ref{fsi}), and may be trivial.
By the last assertion of Theorem~\ref{idempotence f and e},
$\sbA/\theta$ is not a retract of $\sbA$, unless
$\sbA$ is odd (i.e., $\theta=\textup{id}_A$).
There is no further constraint on $[\neg(f^2),f^2]$, while the $e,\leqslant$ reduct
of $\sbA/\theta$ may be any chain with a self-inverting antitone bijection, having a fixed point.
In fact, $\sbA$ is a directed union of algebras, each of which results from $[\neg(f^2),f^2]$ by taking
a rigorously compact two-point extension finitely many times.}

\emph{This largely reduces the study of irreducible De Morgan monoids to the anti-idempotent case.}\hfill{$\Box$}
\end{remk}

We depict below the two-element Boolean algebra $\mathbf{2}$ ($=\sbS_2$), the three-element Sugihara monoid $\sbS_3$, and two $0$--generated four-element
De Morgan monoids, $\sbC_4$ and $\sbD_4$.  In each case, the labeled Hasse diagram determines the structure, in view of Lemma~\ref{bounds},
Theorem~\ref{dm fsi rigorously compact} and the definitions.  That $\sbC_4$ and $\sbD_4$ are indeed De Morgan monoids was noted long ago in the relevance logic literature,
e.g., \cite{Mey83,Mey86}.
All four algebras are simple, by Lemma~\ref{fsi si simple}(\ref{simple}).

{\tiny

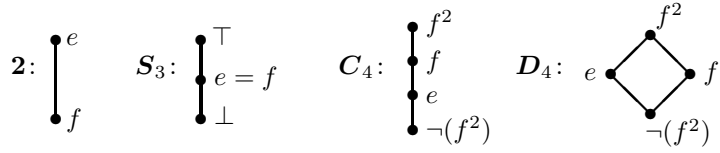
\thicklines
\begin{center}
\begin{picture}(80,60)(-28,51)

\put(-105,63){\line(0,1){30}}
\put(-105,63){\circle*{4}}
\put(-105,93){\circle*{4}}

\put(-101,91){\small $e$}
\put(-101,60){\small $f$}

\put(-122,80){\small $\mathbf{{2}\colon}$}

%
%

\put(-50,78){\circle*{4}}
\put(-50,63){\line(0,1){30}}
\put(-50,63){\circle*{4}}
\put(-50,93){\circle*{4}}

\put(-46,91){\small ${\top}$}
\put(-45,76){\small ${e}=f$}
\put(-45,61){\small ${\bot}$}

\put(-75,80){\small ${\sbS_3}\colon$}

%
%

\put(30,59){\circle*{4}}
\put(30,59){\line(0,1){39}}
\put(30,72){\circle*{4}}
\put(30,85){\circle*{4}}
\put(30,98){\circle*{4}}

\put(35,96){\small ${f^2}$}
\put(35,82){\small $f$}
\put(35,69){\small ${e}$}
\put(35,56){\small $\neg(f^2)$}

\put(2,80){\small ${\sbC_4}\colon$}

%
%

\put(120,65){\circle*{4}}
\put(135,80){\line(-1,-1){15}}
\put(135,80){\circle*{4}}
\put(105,80){\line(1,-1){15}}
\put(105,80){\circle*{4}}
\put(105,80){\line(1,1){15}}
\put(120,95){\circle*{4}}
\put(135,80){\line(-1,1){15}}

\put(122,99){\small ${f^2}$}
\put(95,78){\small ${e}$}
\put(140,78){\small $f$}
\put(118,55){\small $\neg(f^2)$}

\put(69,80){\small ${\sbD_4}\colon$}

\end{picture}\nopagebreak
\end{center}

}

The next theorem is implicit in the findings of Slaney \cite{Sla85,Sla89} mentioned after Lemma~\ref{integral implies boolean},
but it is easier here to give a self-contained proof.

\begin{thm}\label{0 gen simples}
Let\/ $\sbA$ be a simple\/ $0$--generated De Morgan monoid.  Then\/ $\sbA\cong\mathbf{2}$ or\/ $\sbA\cong\sbC_4$ or\/ $\sbA\cong\sbD_4$\textup{.}
\end{thm}
\begin{proof}
Because $\sbA$ is simple (hence nontrivial) and $0$--generated, $\{e\}$ is not a subuniverse of $\sbA$, so $e\neq f$ and
$e$ has just one strict lower bound in $\sbA$ (Lemma~\ref{fsi si simple}(\ref{simple})).  Suppose $\sbA\not\cong\mathbf{2}$.
As every simple Boolean algebra is isomorphic to $\mathbf{2}$, Lemma~\ref{integral implies boolean} shows that $\sbA$ is not integral.
Equivalently, $f$ is not the least element of $\sbA$, so $f\nleqslant e$.  Then by Theorem~\ref{idempotence f and e}, $\sbA$ is not
idempotent and
$f<f^2$, hence $\neg(f^2)<e$, so
$\neg(f^2)$ is the least
element of $\sbA$, i.e., $f^2$ is the greatest element.  Consequently, $a\bcdw\neg(f^2)=\neg(f^2)$ for all $a\in A$, by Lemma~\ref{bounds}, and
$a\bcdw f^2=f^2$ whenever $\neg(f^2)\neq a\in A$, by Theorem~\ref{dm fsi rigorously compact}.

There are two possibilities for the order: $e<f$ or $e\not\leqslant f$.  If $e\not\leqslant f$, then $e\wedge f<e$, whence $e\wedge f$ is
the extremum $\neg(f^2)$ and, by De Morgan's laws, $e\vee f=f^2$.  Otherwise, $\neg(f^2)<e<f<f^2$.  Either way, $\{\neg(f^2),e,f,f^2\}$
is the universe of a four-element subalgebra of $\sbA$, having no proper subalgebra of its own, so $A=\{\neg(f^2),e,f,f^2\}$, as
$\sbA$ is $0$--generated.  Thus, $\sbA\cong\sbC_4$ if $e<f$, and
$\sbA\cong\sbD_4$ if $e\not\leqslant f$.
\end{proof}

We remark that both $\mathbb{V}(\sbC_4)$ and $\mathbb{V}(\sbD_4)$ are categorically equivalent to the variety $\mathbb{V}(\mathbf{2})$ of all Boolean algebras.
(Equivalently, $\sbC_4$ and $\sbD_4$ are primal algebras, as they generate arithmetical varieties and are finite, simple and lack proper
subalgebras and nontrivial automorphisms; see \cite{FP64a,Hu69,McK96}.)

\begin{thm}\label{omit c4 d4}
A variety\/ $\mathsf{K}$ of De Morgan monoids consists of Sugihara monoids iff it excludes\/ $\sbC_4$ and\/ $\sbD_4$\textup{.}
\end{thm}
\begin{proof}
Necessity is clear.  Conversely, suppose $\sbC_4,\sbD_4\notin\mathsf{K}$ and let $\sbA\in\mathsf{K}$ be SI.
It suffices to show that $\sbA$ is a Sugihara monoid.  Suppose not.  Then, by Theorem~\ref{interval subalgebra}
and Remark~\ref{interpretation}, $\neg(f^2)<f^2$ and
the subalgebra $\sbB$ of $\sbA$ on $[\neg(f^2),f^2]$ is nontrivial, whence the $0$--generated subalgebra $\sbE$ of $\sbA$
is nontrivial.
Recall that
every nontrivial finitely generated algebra of finite type has a simple homomorphic image
\cite[Cor.~4.1.13]{Jon72}.  Let $\sbG$ be a simple homomorphic image of $\sbE$, so $\sbG\in\mathsf{K}$.
By assumption, neither $\sbC_4$ nor $\sbD_4$ is isomorphic to $\sbG$, but $\sbG$ is $0$--generated,
so $\mathbf{2}\cong\sbG$, by Theorem~\ref{0 gen simples}.  Thus, $\mathbf{2}\in\mathbb{HS}(\sbB)$.  Then
$\mathbf{2}$ must inherit from $\sbB$ the anti-idempotent identity $x\leqslant f^2$.
This is false, however, so $\sbA$ is a Sugihara monoid.
\end{proof}

In what follows, some features of $\sbC_4$ will be important.

\begin{lem}\label{pre m}
Let\/ $\sbA$ be a nontrivial square-increasing IRL.
\begin{enumerate}
\item\label{pre m 1}
If\/ $e\leqslant f$ and\/ $a\leqslant f^2$ for all\/ $a\in A$\textup{,}
then\/
$e<f$\textup{.}

\smallskip

\item\label{pre m 0}
If\/ $e<f$ in\/ $\sbA$\textup{,} then\/ $\sbC_4$ can be embedded into\/
$\sbA$\textup{.}
\end{enumerate}
\end{lem}
\begin{proof}
(\ref{pre m 1}) \,Suppose $\sbA$ satisfies $e\leqslant f$ and $x\leqslant f^2$.  Then $\sbA$
is not idempotent, by Corollary~\ref{sug cor}, so $f\neq e$, by Theorem~\ref{idempotence f and e}, i.e.,
$e<f$.

(\ref{pre m 0}) \,Suppose $e<f$ in $\sbA$.  Then $f<f^2$, by Theorem~\ref{idempotence f and e}, i.e., $\neg(f^2)<e$.   Thus, $\{\neg(f^2),e,f,f^2\}$
is closed under $\wedge,\vee$ and $\neg$, and $\neg(f^2)$ is idempotent, by (\ref{square increasing cor2}).
By Lemma~\ref{cube}, $f^2$ is an idempotent upper bound of $e$,
so $f^2\bcdw\neg(f^2)=\neg(f^2)$, by Lemma~\ref{3 conditions}.
Closure of $\{\neg(f^2),e,f,f^2\}$ under fusion follows from these
observations and (\ref{isotone}), so $\sbC_4$ embeds into $\sbA$.
\end{proof}

\begin{thm}\label{slaney onto c4}
\textup{(Slaney \cite[Thm.~1]{Sla89})}
\,Let\/ $h\colon\sbA\mrig\sbB$ be a homomorphism, where\/ $\sbA$ is an FSI \,De Morgan monoid, and\/ $\sbB$ is nontrivial
and\/ $0$--generated.  Then\/ $h$ is an isomorphism or\/ $\sbB\cong\sbC_4$\textup{.}
\end{thm}
\begin{proof}
As $\sbB$ is
$0$--generated, $h$ is surjective.  Suppose $h$ is not an isomorphism.  By the remarks preceding
Lemma~\ref{bounds},
$h(a)=e$ for some $a\in A$ with $a<e$.
By Theorem~\ref{splitting}, $a\leqslant f$, so $h(a)\leqslant h(f)$, i.e., $e
\leqslant f$ in $\sbB$.  As $\sbB$ is $0$--generated but not trivial, it cannot
satisfy $e=f$, so $e<f$ in $\sbB$.  Then $\sbC_4$ embeds into $\sbB$, by Lemma~\ref{pre m}(\ref{pre m 0}),
so $\sbB\cong\sbC_4$, again because $\sbB$ is $0$--generated.
\end{proof}

\section{Minimality}\label{minimality section}

A quasivariety is said to be {\em minimal\/} if it is nontrivial and has no nontrivial proper subquasivariety.  If we say that a variety is {\em minimal\/} (without
further qualification), we mean that it is nontrivial and has no nontrivial proper subvariety.  When we mean instead that it is {\em minimal as a quasivariety}, we shall say
so explicitly, thereby avoiding ambiguity.

\begin{thm}\label{atoms}
The distinct classes\/ $\mathbb{V}(\mathbf{2})$\textup{,} $\mathbb{V}(\sbS_3)$\textup{,} $\mathbb{V}(\sbC_4)$ and\/ $\mathbb{V}(\sbD_4)$
are precisely the minimal varieties of De Morgan monoids.
\end{thm}
\begin{proof}
Each $\sbX\in\{\mathbf{2},\sbS_3,\sbC_4,\sbD_4\}$ is finite and simple, with no proper nontrivial subalgebra, so the nontrivial members
of $\mathbb{HS}(\sbX)$ are isomorphic to $\sbX$.  Thus, the SI members of $\mathbb{V}(\sbX)$ belong to $\mathbb{I}(\sbX)$, by
J\'{o}nsson's Theorem, because $\mathsf{DMM}$ is a
congruence distributive variety.
As varieties are determined by their SI members, this shows that
$\mathbb{V}(\sbX)$ has no proper nontrivial subvariety, and that $\mathbb{V}(\sbX)\neq\mathbb{V}(\sbY)$ for distinct $\sbX,\sbY\in
\{\mathbf{2},\sbS_3,\sbC_4,\sbD_4\}$.  As $\mathbb{V}(\mathbf{2})$ and $\mathbb{V}(\sbS_3)$ are the only minimal varieties of
Sugihara monoids, Theorem~\ref{omit c4 d4} shows that they, together with $\mathbb{V}(\sbC_4)$ and $\mathbb{V}(\sbD_4)$,
are the only minimal subvarieties of $\mathsf{DMM}$.
\end{proof}

\begin{tightcenter}
\begin{picture}(120,170)(-60,-4)

\put(0,20){\circle*{4}}
\put(-16,44){\line(2,-3){16}}
\put(-16,44){\circle*{4}}
\put(-40,44){\line(5,-3){40}}
\put(-40,44){\circle*{4}}
\put(16,44){\line(-2,-3){16}}
\put(16,44){\circle*{4}}
\put(40,44){\line(-5,-3){40}}
\put(40,44){\circle*{4}}

\qbezier(-40,44)(-50,50)(-50,100)
\qbezier(-50,100)(-45,135)(0,140)

\qbezier(40,44)(50,50)(50,100)
\qbezier(50,100)(45,135)(0,140)

\put(0,140){\circle*{4}}
\put(-13,7){\small trivial}
\put(-60,30){\small $ \Vop(\sbC_4) $}
\put(35,30){\small $ \Vop(\sbS_3) $}
\put(-30,52){\small $ \Vop(\sbD_4) $}
\put(8,52){\small $ \Vop(\Alg{2}) $}
\put(-10,145){\small $ \Class{DMM} $}

\end{picture}
\end{tightcenter}

Bergman and McKenzie \cite{BM90} showed that every locally finite congruence modular minimal variety is also minimal as a quasivariety.
Thus, by Theorem~\ref{atoms},
$\mathbb{V}(\mathbf{2})$, $\mathbb{V}(\sbS_3)$, $\mathbb{V}(\sbC_4)$ and $\mathbb{V}(\sbD_4)$ are minimal
as quasivarieties.  (In a sequel paper \cite{MRW2}, we show that $\mathsf{DMM}$ has just 68 minimal subquasivarieties.)
With a view to axiomatizing the varieties in Theorem~\ref{atoms}, consider the following (abbreviated) equations. \enlargethispage{5pt}
\begin{align}
\label{eq:semilinear}&e\leqslant (x\rig y)\vee(y\rig x)\\
\label{eq:S3}&e \leqslant (x \to (y \vee \neg y)) \vee (y \wedge \neg y)\\
\label{eq:D42}&e \leqslant (f^2 \to x) \vee (x \to e) \vee \neg x\\
\label{eq:C41}&x \wedge (x \to f) \leqslant (f \to x) \vee (x \to e) \\
\label{eq:C42}&x \to e \leqslant x \vee (f^2 \to \neg x)
\end{align}
It is shown in \cite{HRT02} that an [I]RL $\sbA$ is semilinear (i.e., a subdirect product of chains) iff it is distributive and satisfies (\ref{eq:semilinear}).

\begin{thm}\
\label{thm:axiomatization}
\begin{enumerate}
\item \label{thm:axiomatization:2} $\Vop(\Alg{2})$ is axiomatized by adding\/ $x \leqslant e$ to the axioms of\/ $\Class{DMM}$\textup{;}

\smallskip

\item \label{thm:axiomatization:S3} $\Vop(\sbS_3)$ by adding\/ $e = f$\textup{,  (\ref{eq:semilinear})} and\/ \textup{(\ref{eq:S3});}

\smallskip

\item \label{thm:axiomatization:D4} $\Vop(\sbD_4)$ by adding\/ $x \leqslant f^2$\textup{,}\, $x \wedge \neg x \leqslant y$ and\/ \textup{(\ref{eq:D42});}

\smallskip

\item \label{thm:axiomatization:C4} $\Vop(\sbC_4)$ by adding\/ $x \leqslant f^2$\textup{,}\, $e \leqslant f$\textup{, (\ref{eq:semilinear}), (\ref{eq:C41})} and\/ \textup{(\ref{eq:C42}).}\,\footnote{\,Of course, (\ref{thm:axiomatization:2}) is well known.  We have not encountered (\ref{thm:axiomatization:S3})--(\ref{thm:axiomatization:C4}) in the literature, but a variant of (\ref{thm:axiomatization:S3}) could be derived from \cite[Cor.~2]{Dun70}.}
\end{enumerate}
\end{thm}

\begin{proof}
Let $\Alg{X} \in \{ \Alg{2}, \sbS_3, \sbC_4, \sbD_4 \}$. It can be verified mechanically that $\Alg{X}$ satisfies the proposed axioms for $\Vop(\Alg{X})$. Let $\Alg{A}$ be an SI De Morgan monoid satisfying the
same
axioms, and let $a$ be the largest element of $\Alg{A}$ strictly below $e$, which exists by Lemma~\ref{fsi si simple}(\ref{si}).
By involution properties, $\neg a$ is the smallest element of $\Alg{A}$ strictly above $f$.  It suffices to show that $\Alg{A} \cong \Alg{X}$.

When $\Alg{X}$ is $\Alg{2}$, this follows from Lemma~\ref{integral implies boolean}, as every SI Boolean algebra is isomorphic to $\Alg{2}$.

If $\Alg{X}$ is $\sbS_3$ or $\sbC_4$, then $\Alg{A}$ is totally ordered
(because it is semilinear, by (\ref{eq:semilinear}), and SI).

Suppose that $\Alg{X} = \sbS_3$. In $\Alg{A}$, since $e=f$, we have $a < e < \neg a$, and there is no other element in the interval $[a,\neg a]$. We claim, moreover, that $\neg a$ has no strict
upper bound in $\sbA$.
Suppose, on the contrary, that $\neg a<b\in A$. By (\ref{eq:S3}) and since $e$ is join-prime (Lemma~\ref{fsi si simple}(\ref{prime})),
we have $e \leqslant b \to (a \vee \neg a) $ or $e \leqslant a \wedge \neg a$.  But $a \wedge \neg a = a < e$, so by (\ref{t order}), $ b \leqslant a \vee \neg a = \neg a$, a contradiction. This vindicates the above claim.
By involutional symmetry, $a$ has no strict lower bound in $\sbA$.  As $\sbA$ is totally ordered, this shows that
$A = \{a,e,\neg a\}$.  Then
$\Alg{A} \cong \sbS_3$, in view of Lemma~\ref{bounds}.

We may now assume that $\Alg{X}$ is $\sbC_4$ or $\sbD_4$, so $\sbA$
satisfies $\textup{$\neg(f^2)\leqslant x\leqslant f^2$}$ and is therefore
rigorously compact (Theorem~\ref{dm fsi rigorously compact})
and not idempotent (Corollary~\ref{sug cor}), whence $f<f^2$ and $f\nleqslant e$ in $\sbA$ (Theorem~\ref{idempotence f and e}).

Suppose $\Alg{X} = \Alg{D_4}$.  By assumption, $b \wedge \neg b = \neg(f^2)$ for any $ b \in A$.
If $e < f$, then $e = e \wedge f = \neg(f^2)$, i.e., $e$ is the bottom element of $\Alg{A}$, forcing $\Alg{A}$ to be trivial (see the remarks before Lemma~\ref{bounds}).  This contradiction shows that $e$ and $f$ are incomparable in $\Alg{A}$.

As $a$ is the greatest strict lower bound of $e$, we now have $a < f$, by Corollary~\ref{subcover}.
Then $a = e \wedge f = \neg(f^2)$ and,
by involution properties,
no element lies strictly between $f$ and $f^2$.
Suppose $b \in A$, with $\neg(f^2) < b < f$.
By (\ref{eq:D42}),
$$e \leqslant (f^2 \to \neg b) \vee (\neg b \to e) \vee b.$$
Since $\Alg{A}$ is rigorously compact and $\neg b \neq f^2$, we have $f^2 \to \neg b = \neg(f^2)$. So, because $e$ is join-prime, $e \leqslant \neg b \to e$ or $e \leqslant b$. The last disjunct is false, for otherwise $e \leqslant b < f$. Therefore, $\neg b \leqslant e$, i.e., $f \leqslant b$, contrary to assumption. Thus, no element of $\sbA$ lies strictly between $\neg(f^2)$ and $f$ and, by involution properties, no element lies strictly between $e$ and $f^2$.  It follows that $A = \{ \neg(f^2), e,f,f^2 \}$, in view of Theorem~\ref{splitting}.
In this case, $\Alg{A} \cong \sbD_4$.

Lastly, suppose $\Alg{X} = \sbC_4$. Note that $\sbC_4$ embeds into $\Alg{A}$, by Lemma~\ref{pre m}.
As $ a < e $, it follows from (\ref{eq:C42}) that
$$ e \leqslant a \to e \leqslant a\vee(f^2 \to \neg a), $$
but $e$ is join-prime and $e\nleqslant a$, so
$f^2 \leqslant \neg a$, whence $a = \neg(f^2)$.
Thus, no element of $\Alg{A}$ lies strictly between $\neg(f^2)$ and $e$, nor strictly between $f$ and $f^2$.

Suppose, with a view to contradiction, that $ b \in A \setminus \{ \neg(f^2), e, f, f^2 \} $.
By the previous paragraph and since $\Alg{A}$ is totally ordered, $ e < b < f $.
Then $ e \leqslant b \to f $,
so by (\ref{eq:C41}),
$$ e \leqslant b \wedge (b \to f) \leqslant (f \to b) \vee (b \to e). $$
Now join-primeness of $e$ gives
$f \leqslant b$ or $b \leqslant e$, a contradiction, so
$\Alg{A} \cong \sbC_4$.
\end{proof}

Theorem~\ref{atoms} says, in effect, that for each axiomatic consistent extension $\mathbf{L}$ of $\mathbf{R}^\mathbf{t}$, there exists
$\sbB\in\{\mathbf{2},\sbS_3,\sbC_4,\sbD_4\}$ such that the theorems of $\mathbf{L}$ all take values $\geqslant e$
on any evaluation of their variables in $\sbB$.
Postulates for the four maximal consistent axiomatic extensions of $\Logic{R^t}$ follow systematically from Theorem~\ref{thm:axiomatization}.
For example,
(\ref{eq:semilinear}) becomes the axiom $\textup{$(p \to q) \vee (q \to p)$}$, while (\ref{eq:C42}) becomes $(p \to \mathbf{t}) \to (p \vee(\mathbf{f}^2 \to \neg p))$.

\section{Relevant Algebras}\label{relevant algebras}

The relevance logic literature is equivocal
as to the precise definition of a De Morgan monoid.  Our Definition~\ref{de morgan monoid definition}
conforms with Dunn and Restall \cite{DR01}, Meyer and Routley \cite{MR72,RM73}, Slaney \cite{Sla85} and Urquhart \cite{Urq84}, yet other papers by some of the same
authors entertain a discrepancy.  In all sources, the neutral element of a De Morgan monoid $\sbA$ is assumed to exist but, in \cite{Sla89,Sla91,Sla93} for
instance, it is not distinguished, i.e., the symbol for $e$ (and likewise $f$) is absent from the signature of $\sbA$.  That locally innocuous convention has
global effects: it would prevent $\mathsf{DMM}$ from being a variety, as it would cease to be closed under subalgebras, and the tight correspondence
between axiomatic extensions of $\mathbf{R}^\mathbf{t}$ and subvarieties of $\mathsf{DMM}$ would
disappear.\footnote{\,The meanings of statements about `$n$--generated De Morgan monoids' would also change.
For instance, \cite[Thm.~5]{Sla89} says that every FSI De Morgan monoid on one idempotent generator is finite, but this
is false when $e$ is distinguished, as the proof of \cite[Thm.~6]{Sla89}
makes clear.}

This may explain why we have found in the literature no
analysis of the
subvariety lattice of $\mathsf{DMM}$
(despite interest in the problem discernable in
\cite{Mey83,Mey86}), and in particular no
statement of Theorem~\ref{atoms}, identifying the only four maximal consistent axiomatic extensions
of $\mathbf{R}^\mathbf{t}$ (although the algebras defining these extensions were well known to relevance logicians).

The practice of not distinguishing neutral elements stems from
the formal system $\mathbf{R}$ of Anderson and Belnap \cite{AB75}, which differs from
$\mathbf{R}^\mathbf{t}$ only in that it lacks the sentential constant $\mathbf{t}$ (corresponding to $e$) and its postulates.
The omission of constants from $\mathbf{R}$ produces a desirable
{\em variable sharing principle\/} for `relevant' implication:
\[
\textup{if \,$\,\vdash_\mathbf{R}\al\rig\be$, \,then $\al$ and $\be$ have a common variable \,\cite{Bel60}.}
\]
The corresponding claim for $\mathbf{R}^\mathbf{t}$ is false, e.g.,
\[
\textup{$\,\vdash_{\mathbf{R}^\mathbf{t}}\,\mathbf{t}\rig (p\rig p)$ \ and \ $\,\vdash_{\mathbf{R}^\mathbf{t}}\,(p\wedge\mathbf{t})\rig(\mathbf{t}\vee q)$.}
\]
\begin{defn}
\textup{A \emph{relevant algebra} is an algebra $\langle A; \bdot, \wedge, \vee, \neg  \rangle$ such that
$\langle A; \bdot \rangle$ is a commutative semigroup, $\langle A; \wedge, \vee \rangle $ is a distributive lattice and
\begin{align*}
& \neg \neg a = a \leqslant a \bdot a,\\
& \textup{$a \leqslant b$ \,iff\, $\neg b \leqslant \neg a$,}\\
& \textup{$a \bdot b \leqslant c$ \,iff\, $a \bdot \neg c \leqslant \neg b$,}\\
& \textup{$a \leqslant a \bdot (\neg (b \bdot \neg b) \wedge \neg(c \bdot \neg c))$,}
\end{align*}
for all $a,b,c \in A$.
The class of all relevant algebras is denoted by $\mathsf{RA}$.}
\end{defn}
The two defining postulates of $\mathsf{RA}$ that are not pure equations can be paraphrased easily as equations, so $\mathsf{RA}$ is a variety.
It is congruence distributive (since its members have lattice reducts) and congruence permutable (see for instance \cite[Prop.~8.3]{vAR04}).

The main motivation for $\mathsf{RA}$ is that it algebraizes the logic $\mathbf{R}$.
The algebraization process for
$\mathbf{R}^\mathbf{t}$ and $\mathsf{DMM}$
carries over verbatim to $\mathbf{R}$ and $\mathsf{RA}$,
provided we use (\ref{t eq}) as a formal device for eliminating all
mention of $e$.
Further work on relevant algebras can be found in \cite{Dzi83,FR90,Mak71,Mak73,Raf11,RS,Swi95,Swi99}.

Because $\mathsf{RA}$ \emph{is} closed under subalgebras, its study accommodates the variable sharing principle of relevance logic, without sacrificing the benefits of
accurate algebraization.  For the algebraist, however, $\mathsf{RA}$ has some forbidding features.
It
lacks the congruence extension property (CEP), for instance, as does its class of finite members (see
\cite[p.\,289]{CD99}), whereas $\mathsf{DMM}$ has the CEP.  Also, De Morgan monoids have much in common with
abelian groups
(residuals being a partial surrogate for multiplicative inverses), but relevant algebras are less intuitive, being semigroup-based, rather than
monoid-based.

The following facts are therefore noteworthy.

\begin{thm}\label{thm:fgenRAisDMM}\
\begin{enumerate}
\item\label{ra subreducts}
$\mathsf{RA}$ coincides with the class of all\/ $e$--\textup{free subreducts} of De Morgan monoids (i.e., all subalgebras of reducts\/ $\langle A;\bcdw,\wedge,\vee,\neg\rangle$
of De Morgan monoids\/ $\sbA$\textup{).}

\smallskip

\item\label{ra fg reducts}
If a relevant algebra is finitely generated, then it is the\/ $e$--free\/ \emph{reduct} $\langle A;\bcdw,\wedge,\vee,\neg\rangle$ of a De Morgan monoid\/ $\sbA$\textup{.}
In this case, the unique neutral element of\/ $\sbA$ is
the greatest lower bound of all\/ $a\rig a$\textup{,} where\/ $a$ ranges over any finite generating set for\/ $\langle A;\bcdw,\wedge,\vee,\neg\rangle$\textup{.}
\end{enumerate}
\end{thm}
\noindent
Here, (\ref{ra fg reducts}) is a specialization of \cite[Thm.~5.3]{ORV08}, but it algebraizes a much older logical
result of Anderson and Belnap
\cite[p.\,343]{AB75}, already implicit in the proof of \cite[Lem.~2]{AB59}.
We can infer (\ref{ra subreducts}) from (\ref{ra fg reducts}), as every algebra embeds into an ultraproduct
of finitely generated subalgebras of itself (or see \cite[Cor.~4.11]{HR07}).

Theorem~\ref{thm:fgenRAisDMM}(\ref{ra subreducts}) reflects the fact that the $\mathbf{t}$--free fragment of $\,\vdash_{\mathbf{R}^\mathbf{t}}$ is just
$\,\vdash_\mathbf{R}$, so there is a smooth passage from either system to the other.  In particular, the variable sharing principle holds for the
$\mathbf{t}$--free formulas of $\mathbf{R}^\mathbf{t}$.

The $e$--free reduct $\langle A;\bcdw,\wedge,\vee,\neg\rangle$ of a De Morgan monoid $\sbA$ shall be denoted by $\sbA^-$.
Also, if $\Class{K}$ is a class
of De Morgan monoids, then $\Class{K}^-$ shall denote the class of $e$--free reducts of the members of $\Class{K}$.
In this case, on general grounds,
\begin{equation}\label{eq:reducts}
\Vop(\Class{K})^- \subseteq \Vop(\Class{K}^-).
\end{equation}
Indeed, every equation satisfied by $\Class{K}^-$
is an $e$--free identity of $\Class{K}$,
and therefore of $\Vop(\Class{K})$,
and therefore of $\Vop(\Class{K})^-$.
Because a De Morgan monoid and its $e$--free reduct have the same congruence lattice,
we also obtain:

\begin{lem}
\label{thm:isomorphicCongruences}
A De Morgan monoid\/ $\Alg{A}$ is SI \,[resp.\ FSI; simple] \,iff the same is true of\/ $\Alg{A}^-$\textup{.}
\end{lem}

Crucially, however, subalgebras
of the $e$--free reduct of a De Morgan monoid need not contain $e$, and they need not be reducts of De Morgan monoids themselves,
unless they are finitely generated.
For instance, the free $\aleph_0$--generated relevant algebra is such a subreduct, and it lacks a neutral element,
because the variable sharing principle rules out theorems of $\mathbf{R}$ of the form $\alpha\rig (p\rig p)$
whenever $p$ is a variable not occurring in the formula $\alpha$.

Still, because of Theorem~\ref{thm:fgenRAisDMM}, it is often easiest to obtain a result about relevant algebras indirectly,
via
a more swiftly established
property of De Morgan monoids.  This is exemplified below in Corollary~\ref{thm:FGRAisBounded}, and more strikingly in Theorem~\ref{thm:minimalvarietiesRA}.
(We extend our use of the terms `bounded' and `rigorously compact' to relevant algebras in the obvious way, noting that the existence of a neutral element is not needed in the proof of Lemma~\ref{rigorously compact}.)

\begin{cor}
\label{thm:FGRAisBounded}
\textup{(\cite{Swi99})} Every finitely generated relevant algebra $\Alg{A}$ is bounded.
\end{cor}

\begin{proof}
By Theorem~\ref{thm:fgenRAisDMM}(\ref{ra fg reducts}), $\Alg{A}$ is a reduct of a De Morgan monoid with
the same finite generating set, so $\Alg{A}$ is bounded,
by Theorem~\ref{bounded}.
\end{proof}

In contrast with this argument, the only published proof of Corollary~\ref{thm:FGRAisBounded},
viz.\ \cite[Prop.~5]{Swi99}, is quite complicated.
(For one generator, the indicated bounds are built up using all six of the inequivalent implicational
one-variable formulas of $\mathbf{R}$, determined in \cite{Mey70}.)
The result is attributed in \cite{Swi99} to Meyer and to Dziobiak (independently).

\begin{cor}
\label{thm:booleansubalgebra}
Every nontrivial relevant algebra\/ $\Alg{A}$ has a copy of\/ $\Alg{2}^-$ as a subalgebra.
\end{cor}

\begin{proof}
Let $\sbB$ be the subalgebra of $\sbA$ generated by an arbitrary pair
of distinct elements of $A$.
By Corollary~\ref{thm:FGRAisBounded},
$\Alg{B}$ has
(distinct) extrema $\bot,\top$.
By Theorem~\ref{thm:fgenRAisDMM}(\ref{ra fg reducts}),
$\sbB$ is the $e$--free reduct of a De Morgan monoid, so
by Lemma~\ref{bounds},
$\{ \bot, \top \}$ is the universe of a subalgebra of $\Alg{B}$, isomorphic to $\Alg{2}^-$.
\end{proof}

Clearly, when a Boolean algebra $\Alg{A}$ is thought of as an integral De Morgan monoid, it has the same term operations as its $e$--free reduct $\Alg{A}^-$, because $e$ is definable as $x \to x$.
Thus,
the variety
of Boolean algebras can be identified with $\Vop(\Alg{2}^-) = \Qop(\Alg{2}^-)$.

\begin{cor}
\label{thm:booleanisminimal}
Boolean algebras constitute the smallest nontrivial (quasi) variety of relevant algebras.
\end{cor}

This reconfirms, of course, that classical propositional logic
is the largest consistent extension of $\Logic{R}$.
With Theorem~\ref{thm:fgenRAisDMM}(\ref{ra fg reducts}),
it also yields the following.

\begin{thm}\label{surjection}
There is a join-preserving (hence isotone) surjection from the lattice of subvarieties of\/ $\mathsf{DMM}$ to that of\/ $\mathsf{RA}$\textup{,}
defined by\/ $\Class{K}\mapsto\Vop(\Class{K}^-)$\textup{.}

Moreover, this map remains surjective when its domain is restricted to the varieties that contain\/ $\mathbf{2}$\textup{,} together
with the trivial variety.
\end{thm}
\begin{proof}
Preservation of joins follows from J\'{o}nsson's Theorem and the following two facts:
(i)~an ultraproduct of reducts of members of a class $\mathsf{C}$ is the reduct of a (corresponding) ultraproduct of members of $\mathsf{C}$, and (ii)~an
ultra\-product of members of the join of two varieties belongs to one of the two varieties.
To prove surjectivity, let $\Class{L}$ be a variety of relevant algebras, and $\Class{L}_\textup{FG}$ its class
of finitely generated members.  By Theorem~\ref{thm:fgenRAisDMM}(\ref{ra fg reducts}), each $\sbA\in\Class{L}_\textup{FG}$ is
the $e$--free reduct of a unique De Morgan monoid $\sbA^+$.
Let
$\textup{$\Class{M}=\{\sbA^+:\sbA\in\Class{L}_\textup{FG}\}$}$.
Then $\Class{L}_\textup{FG}=\Class{M}^-\subseteq\Vop(\Class{M})^-$, while (\ref{eq:reducts}) shows that
$\Vop(\Class{M})^-\subseteq\Vop(\Class{L}_\textup{FG})$.  Thus, $\Class{L}=\Vop(\Vop(\Class{M})^-)$, as varieties are
determined by their finitely generated members.  If $\Class{L}$ is nontrivial, then $\mathbf{2}^-\in\Class{L}_\textup{FG}$, by Corollary~\ref{thm:booleansubalgebra}, so $\mathbf{2}\in\Class{M}$.
\end{proof}

Whereas the above argument about joins would apply in any context where the indicated reduct class is a congruence
distributive variety, the surjectivity of the (restricted) function in Theorem~\ref{surjection} is a special feature
of relevant algebras, reliant on Theorem~\ref{thm:fgenRAisDMM}(\ref{ra fg reducts}).  The restricted function is not
injective, however.  Indeed, J\'{o}nsson's Theorem shows that
$\mathbb{V}(\mathbf{2},\sbS_{2n+1})\subsetneq\mathbb{V}(\sbS_{2n},\sbS_{2n+1})$
for all integers $n>1$, but these two varieties have the same image
under the map $\mathsf{K}\mapsto\mathbb{V}(\mathsf{K}^-)$, because $\sbS_{2n}^-$
embeds into $\sbS_{2n+1}^-$ (although $\sbS_{2n}$ does not embed into $\sbS_{2n+1}$).\,\footnote{\,The
function sending a subvariety $\mathsf{W}$ of $\mathsf{RA}$ to the variety generated by the
De Morgan monoids whose $e$--free reducts lie in $\mathsf{W}$ is an injective join-preserving right-inverse for the
function in Theorem~\ref{surjection}, but it is not surjective.}

This failure of injectivity limits the usefulness of the above function when we analyse the
subvariety lattice of $\mathsf{RA}$.
Nevertheless, we can already derive \'{S}wirydowicz's description of the lower
part of that lattice
by a mathematically simpler argument, based on the situation for De Morgan
monoids.  In particular, we avoid use of the complex ternary relation semantics for $\mathbf{R}$ (see \cite{RM73}),
employed in \cite{Swi95}.

\begin{thm}
\label{thm:minimalvarietiesRA}
\textup{(\cite[Thm.~12]{Swi95})} \,$\Vop(\sbS_3^-)$\textup{,} $\Vop(\sbC_4^-) $ and\/ $\Vop(\sbD_4^-)$ are exactly the covers of\/ $\Vop(\Alg{2}^-)$ in the subvariety lattice of\/ $\Class{RA}$\textup{.}
\end{thm}

\begin{proof}
Let $\sbX\in\{\sbS_3^-, \sbC_4^-,\sbD_4^-\}$, so $\sbX$ is simple (by Lemma~\ref{thm:isomorphicCongruences})
and $\sbX$ has just one nontrivial proper subalgebra, which is isomorphic to $\Alg{2}^-$.
Then
every SI member of $\Vop(\Alg{X})$ is isomorphic to $\Alg{2}^-$ or to $\Alg{X}$, by J\'{o}nsson's Theorem (cf.\ the proof of Theorem~\ref{atoms}).
So,
there are no subvarieties of $\Class{RA}$ strictly between $\Vop(\Alg{2}^-)$ and $\Vop(\Alg{X})$,
and $\Vop(\Alg{X}) \neq \Vop(\Alg{Y})$ for $\Alg{X}\neq \Alg{Y} \in \{ \sbS_3^-, \sbC_4^-, \sbD_4^- \}$.

Conversely,
let $\Class{K}$ be a subvariety of $\Class{RA}$\textup{,} not consisting entirely of Boolean algebras.  We must show that
$\Vop(\Alg{X}) \subseteq \Class{K}$ for some $\Alg{X} \in \{ \sbS_3^-, \sbC_4^-, \sbD_4^- \}$.

As
$\Vop(\Alg{2}^-)\subsetneq\mathsf{K}$ are varieties,
there exists
a finitely generated SI algebra $\textup{$\Alg{A} \in \Class{K}\setminus\Vop(\Alg{2}^-)$}$.
Now
$\Alg{A}$ is the $e$--free reduct of some $\Alg{A}^+\in\Class{DMM}$, by Theorem~\ref{thm:fgenRAisDMM}(\ref{ra fg reducts}),
and $\Alg{A}^+$ is SI (by Lemma~\ref{thm:isomorphicCongruences}) and finitely generated.
By (\ref{eq:reducts}),
$$\Vop(\Alg{A}^+)^- \subseteq \Vop(\Alg{A})
\subseteq \Class{K},$$
so it suffices to show that one
of $\sbS_3$, $\sbC_4$ or $\sbD_4$ belongs to $\Vop(\Alg{A}^+)$.

Suppose $\sbC_4,\sbD_4\notin\Vop(\Alg{A}^+)$.  Then
$\Alg{A}^+$ is
a Sugihara monoid, by Theorem~\ref{omit c4 d4}.
As $\Alg{A}^+$ is SI, finitely generated, and not a Boolean algebra, it is isomorphic to $\sbS_n$ for some
$n\geq 3$, by Theorem~\ref{sug thm}.  Then $\sbS_3\in\mathbb{H}(\Alg{A}^+)\subseteq \mathbb{V}(\Alg{A}^+)$
(by the remarks preceding Corollary~\ref{osm varieties}), completing the proof.
\end{proof}

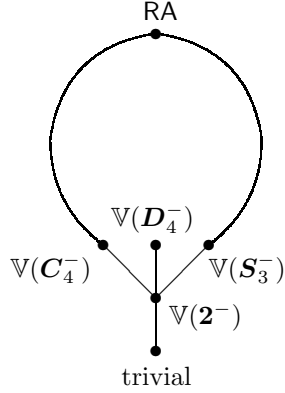
\begin{tightcenter}
\begin{picture}(100,168)(-50,-3)

\put(0,20){\circle*{4}}
\put(0,40){\line(0,-1){20}}
\put(0,40){\circle*{4}}
\put(0,60){\line(0,-1){20}}
\put(0,60){\circle*{4}}
\put(-20,60){\line(1,-1){20}}
\put(-20,60){\circle*{4}}
\put(20,60){\line(-1,-1){20}}
\put(20,60){\circle*{4}}

\qbezier(-20,60)(-40,75)(-40,100)
\qbezier(-40,100)(-35,135)(0,140)

\qbezier(20,60)(40,75)(40,100)
\qbezier(40,100)(35,135)(0,140)
\put(0,140){\circle*{4}}

\put(-13,7){\small trivial}
\put(5,30){\small $ \Vop(\Alg{2}^-) $}
\put(-55,48){\small $ \Vop(\sbC_4^-) $}
\put(20,48){\small $ \Vop(\sbS_3^-) $}
\put(-17,67){\small $ \Vop(\sbD_4^-) $}
\put(-5,145){\small $ \Class{RA} $}

\end{picture}
\end{tightcenter}

This means that the logics algebraized by\/ $\Vop(\sbS_3^-)$\textup{,} $\Vop(\sbC_4^-)$ and\/ $\Vop(\sbD_4^-)$ are exactly the maximal non-classical axiomatic extensions of\/
$\Logic{R}$\textup{,} as was observed in \cite{Swi95}.
The proof of Theorem~\ref{thm:minimalvarietiesRA} in \cite{Swi95} relies on a lemma,
which says that every bounded SI relevant algebra is rigorously compact \cite[Lem.~8]{Swi95}.  In \cite{Swi95}, the proof of the lemma uses the ternary relation semantics for $\mathbf{R}$.  As the lemma is itself of some interest, we supply an algebraic justification of it here.  The key to the argument is that the subalgebras of FSI relevant algebras are still FSI, but that fact is concealed by the failure of the CEP and the lack of an obvious analogue for Lemma~\ref{fsi si simple}(\ref{fsi}) in $\mathsf{RA}$.  One way to circumvent these difficulties is to extend the concept of deductive filters to relevant algebras.

\begin{defn}\label{deductive filter}
\textup{A subset $F$ of a
relevant algebra $\Alg{A}$
is called a
\emph{deductive filter}
of $\Alg{A}$ if $F$ is a lattice filter of $\langle A;\wedge,\vee\rangle$ and
\[
\textup{$|a|\seteq a\rig a \in F$ for all $a \in A$.}
\]}\end{defn}
Clearly, the set of deductive filters of $\Alg{A}$ is closed under arbitrary intersections and under unions of non-empty directed subfamilies,
so it is both an algebraic closure system over $A$ and the universe of an algebraic lattice $\boldsymbol{\mathit{DFil}}\,\Alg{A}$, ordered by
inclusion.  We denote by $\mathit{DFg}^\sbA X$ the smallest deductive filter of $\Alg{A}$ containing $X$, whenever $X \subseteq A$.  Thus,
the compact elements of $\boldsymbol{\mathit{DFil}}\,\Alg{A}$ are just the finitely generated deductive filters of $\sbA$, i.e., those of the
form $\mathit{DFg}^\sbA X$ for some finite $X\subseteq A$.

The deductive filters of a relevant algebra $\sbA$ are just the subsets that contain all $\sbA$--instances of the axioms of $\mathbf{R}$
and are closed under the inference rules---modus ponens and adjunction---of $\mathbf{R}$.
(This is easily verified, using (\ref{t eq}) and Theorem~\ref{thm:fgenRAisDMM}(\ref{ra subreducts}).)
Therefore, by the theory of algebraization \cite[Thm.~5.1]{BP89}, and since $\mathsf{RA}$
is a variety, we have
\[
\textup{$\boldsymbol{\mathit{DFil}}\,\Alg{A} \cong \boldsymbol{\mathit{Con}}\,\Alg{A}$, for all $\sbA\in\mathsf{RA}$.}
\]

\begin{thm}
\label{thm:filtergenerationRA}
Let\/ $\Alg{A}$ be a relevant algebra, with\/ $a,b \in A$\textup{.}  Then
\begin{enumerate}
\item \label{mods}
$\left|\left|a\right|\wedge\left|b\right|\right|\leqslant\left|a\right|\wedge\left|b\right|$\textup{;}

\smallskip

\item \label{thm:filtergenerationRA:single}
$\mathit{DFg}^\Alg{A}\{a\} = \{ c \in A : a \wedge |d| \leqslant c \text{ for some } d \in A \}$\textup{;}

\smallskip

\item \label{thm:filtergenerationRA:meet}
$\mathit{DFg}^\sbA\{a\} \cap \mathit{DFg}^\sbA\{b\} = \mathit{DFg}^\sbA\{a \vee b\}$\textup{.}
\end{enumerate}
\end{thm}

\begin{proof}
(\ref{mods}) \,It suffices, by Theorem~\ref{thm:fgenRAisDMM}(\ref{ra subreducts}) and (\ref{t eq}), to show that $e\leqslant \left|a\right|\wedge\left|b\right|$ in
any De Morgan monoid that contains $\sbA$ as a subreduct.  And this follows from (\ref{t laws}).

(\ref{thm:filtergenerationRA:single}) \,Let $F =  \{ c \in A : a \wedge |d| \leqslant c \text{ for some } d \in A \}$. Then $a \in F$, since $a \wedge |a| \leqslant a$.
We claim that $F$ is a deductive filter of $\Alg{A}$.
Clearly, $F$ is upward closed.
Suppose $c,c' \in F$, so there exist $d,d' \in A$ such that $a \wedge |d| \leqslant c $ and $a \wedge |d'| \leqslant c'$. Then
$ c \wedge c' \geqslant a \wedge |d| \wedge |d'| \geqslant a \wedge ||d| \wedge |d'||$,
by (\ref{mods}),
so $c \wedge c' \in F$.
Also, for any $d \in A$, we have $a \wedge |d| \leqslant |d|$, so $|d| \in F$.
This vindicates the claim.
It remains to show that $F$ is the \emph{smallest} deductive filter of $\sbA$ containing $a$.  So, let $G\in\mathit{DFil}\,\sbA$, with $a\in G$, and let $c \in F$.  Choose $d \in A$ with
$a \wedge |d| \leqslant c$.  Since $a,|d| \in G$, we have $ a \wedge |d| \in G$, whence $c \in G$, as required.

(\ref{thm:filtergenerationRA:meet})
\,Certainly, $ a \vee b \in \mathit{DFg}^\sbA\{a\} \cap \mathit{DFg}^\sbA\{b\}$, as $ a,b \leqslant a \vee b$.  Now
suppose $ c \in \mathit{DFg}^\sbA\{a\} \cap \mathit{DFg}^\sbA\{b\}$.  Choose $d, d' \in A$, with $ a \wedge |d| \leqslant c $ and $ b \wedge |d'| \leqslant c$.
Then $ c \geqslant a \wedge |d| \wedge |d'|,\,b \wedge |d| \wedge |d'|$, so \enlargethispage{5pt}
\begin{align*}
c &\geqslant (a \wedge |d| \wedge |d'|)\vee(b \wedge |d| \wedge |d'|) \\
&= (a \vee b) \wedge (|d| \wedge |d'|) \textup{ \ (by distributivity) }\\
&\geqslant  (a \vee b) \wedge ||d|\wedge|d'|| \textup{ \ (by (\ref{mods})).}
\end{align*}
Thus, $ c \in \mathit{DFg}^\sbA\{a \vee b\}$ and the result follows.
\end{proof}

\begin{cor}
\label{thm:subalgebrasAreFSI}
The class of FSI relevant algebras
is closed under subalgebras (and ultraproducts).
\end{cor}

\begin{proof}
Let $\Alg{A} \in \Class{RA}$.
Clearly, $\mathit{DFg}^\Alg{A}\{a_1, \dots, a_n\} = \mathit{DFg}^\Alg{A}\{a_1 \wedge \,\dots\, \wedge a_n\}$ for all $\textup{$a_1,\dots,a_n\in A$}$,
so every finitely generated deductive filter of $\Alg{A}$ is principal.
Therefore, by Theorem~\ref{thm:filtergenerationRA}(\ref{thm:filtergenerationRA:meet}), the intersection of any two compact (i.e., finitely generated)
elements of
$\boldsymbol{\mathit{DFil}}\,\Alg{A}$ is compact.  The same applies to the lattice $\boldsymbol{\mathit{Con}}\,\Alg{A}$, as it is isomorphic
to $\boldsymbol{\mathit{DFil}}\,\Alg{A}$ (and since lattice isomorphisms between complete lattices preserve compactness).
Now the result follows from the well known theorem below.
\end{proof}

\begin{thm}
\textup{(\cite{BP86a})} \,In any congruence distributive variety\/ $\mathsf{K}$\textup{,} the following conditions are equivalent.
\begin{enumerate}
\item
For any\/ $\sbA\in\mathsf{K}$\textup{,} the intersection of any two compact (i.e., finitely generated) congruences of\/ $\sbA$ is compact.

\smallskip

\item
$\mathsf{K}_\textup{FSI}$ is closed under\/ $\mathbb{S}$ and\/ $\mathbb{P}_\mathbb{U}$ (i.e., it is a universal class).
\end{enumerate}
\end{thm}

Finally, as promised, a slight generalization of \cite[Lem.\,8]{Swi95} follows easily from Cor\-ollary~\ref{thm:subalgebrasAreFSI}:

\begin{thm}\label{ra bdd fsi is rc}
Every bounded FSI relevant algebra\/ $\sbA$ is rigorously compact.
\end{thm}
\begin{proof}
Let $\bot,\top$ be the extrema of $\sbA$, and consider $\bot\neq a\in A$.  We must show that $\top\bcdw a=\top$.
Observe that $\sbB\seteq\boldsymbol{\mathit{Sg}}^\sbA \{\bot,a,\top\}$ is FSI,
by Corollary~\ref{thm:subalgebrasAreFSI}.
As $\sbB$ is
finitely generated, it is a reduct of a (bounded) De Morgan monoid $\sbB^+$, by Theorem~\ref{thm:fgenRAisDMM}(\ref{ra fg reducts}), which
is also FSI, by Lemma~\ref{thm:isomorphicCongruences}.  Now $\sbB^+$ is rigorously compact, by
Theorem~\ref{dm fsi rigorously compact}, so $\top\bcdw a=\top$.
\end{proof}

\begin{cor}
Every finitely generated subalgebra of an FSI relevant algebra is rigorously compact.
\end{cor}
\begin{proof}
Use Corollaries~\ref{thm:FGRAisBounded} and \ref{thm:subalgebrasAreFSI} and Theorem~\ref{ra bdd fsi is rc}.
\end{proof}

\end{document}